\documentclass[11pt]{amsart}
\usepackage{amscd}
\usepackage[arrow,matrix]{xy}
\usepackage{graphicx}
\usepackage{amsmath}
\usepackage{amsrefs}
\usepackage{mathtools}
\usepackage{xcolor}
\usepackage{listings}

\definecolor{codegreen}{rgb}{0,0.6,0}
\definecolor{codegray}{rgb}{0.5,0.5,0.5}
\definecolor{codepurple}{rgb}{0.58,0,0.82}
\definecolor{backcolour}{rgb}{0.96,0.96,0.96}

\lstdefinestyle{sagemathstyle}{
	backgroundcolor=\color{backcolour},   
	commentstyle=\color{codegreen},
	keywordstyle=\color{magenta},
	numberstyle=\tiny\color{codegray},
	stringstyle=\color{codepurple},
	basicstyle=\small\ttfamily,
	breakatwhitespace=false,         
	breaklines=true,                 
	captionpos=b,                    
	keepspaces=true,                 
	numbers=left,                    
	numbersep=5pt,                  
	showspaces=false,                
	showstringspaces=false,
	showtabs=false,                  
	tabsize=4,
	frame=single,
	rulecolor=\color{black!30},
	title=\lstname
}
\usepackage{soul}
\usepackage{comment}
\usepackage{latexsym,amssymb}
\usepackage{hyperref}
\hypersetup{hidelinks}
\input xypic
\numberwithin{equation}{section}
\theoremstyle{plain}
\newtheorem{lemma}{Lemma}[section]
\newtheorem{proposition}[lemma]{Proposition}
\newtheorem{theorem}[lemma]{Theorem}
\newtheorem{corollary}[lemma]{Corollary}

\theoremstyle{definition}
\newtheorem{definition}[lemma]{Definition}
\newtheorem{remark}[lemma]{Remark}

\newtheorem{examples}[lemma]{Examples}

\usepackage{tikz}
\usepackage{dsfont}
\definecolor{grey}{RGB}{188,188,188}

\begin{document}
	\newcommand{\R}{{\mathbb R}}
	\newcommand{\C}{{\mathbb C}}
	\newcommand{\F}{{\mathbb F}}
	\renewcommand{\O}{{\mathbb O}}
	\newcommand{\Z}{{\mathbb Z}}
	\newcommand{\N}{{\mathbb N}}
	\newcommand{\Q}{{\mathbb Q}}
	\renewcommand{\H}{{\mathbb H}}
	\newcommand{\X}{{\mathfrak X}}

	\newcommand{\Aa}{{\mathcal A}}
	\newcommand{\Bb}{{\mathcal B}}
	\newcommand{\Cc}{{\mathcal C}}    %configuration space
	\newcommand{\Dd}{{\mathcal D}}
	\newcommand{\Ee}{{\mathcal E}}
	\newcommand{\Ff}{{\mathcal F}}
	\newcommand{\Gg}{{\mathcal G}}    %gauge transformations
	\newcommand{\Hh}{{\mathcal H}}
	\newcommand{\Kk}{{\mathcal K}}
	\newcommand{\Ii}{{\mathcal I}}
	\newcommand{\Jj}{{\mathcal J}}
	\newcommand{\Ll}{{\mathcal L}}    %Loop space
	\newcommand{\Mm}{{\mathcal M}}    %moduli space
	\newcommand{\Nn}{{\mathcal N}}
	\newcommand{\Oo}{{\mathcal O}}
	\newcommand{\Pp}{{\mathcal P}}
	\newcommand{\Qq}{{\mathcal Q}}
	\newcommand{\Rr}{{\mathcal R}}
	\newcommand{\Ss}{{\mathcal S}}
	\newcommand{\Tt}{{\mathcal T}}
	\newcommand{\Uu}{{\mathcal U}}
	\newcommand{\Vv}{{\mathcal V}}
	\newcommand{\Ww}{{\mathcal W}}
	\newcommand{\Xx}{{\mathcal X}}
	\newcommand{\Yy}{{\mathcal Y}}
	\newcommand{\Zz}{{\mathcal Z}}
	
	\renewcommand{\a}{{\mathfrak a}}
	\renewcommand{\b}{{\mathfrak b}}
	\newcommand{\e}{{\mathfrak e}}
	\renewcommand{\k}{{\mathfrak k}}
	\newcommand{\m}{{\mathfrak m}}
	\newcommand{\pg}{{\mathfrak p}}
	\newcommand{\g}{{\mathfrak g}}
	\newcommand{\gl}{{\mathfrak gl}}
	\newcommand{\h}{{\mathfrak h}}
	\renewcommand{\l}{{\mathfrak l}}
	\newcommand{\sm}{{\mathfrak m}}
	\newcommand{\n}{{\mathfrak n}}
	\newcommand{\s}{{\mathfrak s}}
	\renewcommand{\o}{{\mathfrak o}}
	\renewcommand{\u}{{\mathfrak u}}
	\newcommand{\su}{{\mathfrak su}}

	\newcommand{\ssl}{{\mathfrak sl}}
	\newcommand{\ssp}{{\mathfrak sp}}
	\renewcommand{\t}{{\mathfrak t }}
	
	\newcommand{\zt}{{\tilde z}}
	\newcommand{\xt}{{\tilde x}}
	\newcommand{\Ht}{\widetilde{H}}
	\newcommand{\ut}{{\tilde u}}
	\newcommand{\Mt}{{\widetilde M}}
	\newcommand{\Llt}{{\widetilde{\mathcal L}}}
	\newcommand{\yt}{{\tilde y}}
	\newcommand{\vt}{{\tilde v}}
	\newcommand{\Ppt}{{\widetilde{\mathcal P}}}
	\newcommand{\bp }{{\bar \partial}}
	
	\newcommand{\Remark}{{\it Remark}}
	\newcommand{\Proof}{{\it Proof}}
	\newcommand{\ad}{{\rm ad}}
	\newcommand{\Om}{{\Omega}}
	\newcommand{\om}{{\omega}}
	\newcommand{\eps}{{\varepsilon}}
	\newcommand{\Di}{{\rm Diff}}
	
	\newcommand{\Cinf}{C^{\infty}}
	\newcommand{\la}{\langle}
	\newcommand{\ra}{\rangle}
	\newcommand{\half}{\scriptstyle\frac{1}{2}}
	\newcommand{\p}{{\partial}}
	\newcommand{\notsub}{\not\subset}
	\newcommand{\iI}{{I}}               %unit interval [0,1]
	\newcommand{\bI}{{\partial I}}      %boundary of same
	\newcommand{\LRA}{\Longrightarrow}
	\newcommand{\LLA}{\Longleftarrow}
	\newcommand{\lra}{\longrightarrow}
	\newcommand{\LLR}{\Longleftrightarrow}
	\newcommand{\lla}{\longleftarrow}
	\newcommand{\INTO}{\hookrightarrow}
	
	\newcommand{\QED}{\hfill$\Box$\medskip}
	\newcommand{\UuU}{\Upsilon _{\delta}(H_0) \times \Uu _{\delta} (J_0)}
	\newcommand{\bm}{\boldmath}
	\newcommand{\pb}{{\mathbf p}}
	
	\newcommand{\GL}{{\rm GL}}
	\newcommand{\SL}{{\rm SL}}
	\newcommand{\SO}{{\rm SO}}
	\newcommand{\G}{{\rm G_2}}
	\newcommand{\Spin}{{\rm Spin(7)}}
	\newcommand{\Id}{\mathrm{Id}}
	\newcommand{\smallr}{\mathrm{small}}
	\newcommand{\Comm}{{\mathrm Comm}}

	\newenvironment{nouppercase}{%
		\let\uppercase\relax%
		\renewcommand{\uppercasenonmath}[1]{}}{}

	\setcounter{tocdepth}{1}
	
	\title[Minimal unital cyclic $C_\infty$-algebras]{Minimal Unital Cyclic $C_\infty$-Algebras  and the Real  and Rational Homotopy Type of Closed Manifolds}
	
	\author{H\^ong V\^an L\^e}
	\address{Institute of Mathematics of the Czech Academy of Sciences, \v Zitn\'a 25, 115 67 Praha 1, Czech Republic}
	\email{hvle@math.cas.cz}
	\date{July 13, 2026}
	\thanks{The section ``A new proof of Crowley--Nordstr\"om's formality result'' was written in collaboration with Domenico Fiorenza.}
	\begin{abstract}   
Using the notion of isotopy modulo $k$, for $k\in\mathbb N^+$, we
introduce a stratification on the set of minimal $C_\infty$-algebra
enhancements of a finite-dimensional graded commutative algebra $H^*$.
We prove that two such enhancements are $C_\infty$-isotopic if and
only if they are isotopic modulo $k$ for every $k\in\mathbb N^+$.
We define obstruction sets governing the extension of an isotopy
modulo $k$ to an isotopy modulo $(k+1)$ and establish their generalized
additivity.	We prove that if $M$ is a closed $(r-1)$-connected manifold of
dimension $n\leq \ell(r-1)+2,\, r\geq2,\, \ell\geq 4$,
then its real and rational homotopy types are determined by its
cohomology algebra $H^*(M;\mathbb F)$ together with the isotopy class
modulo $(\ell-2)$ of the corresponding minimal unital cyclic
$C_\infty$-algebra enhancement, for $\mathbb F=\mathbb R$ and
$\mathbb F=\mathbb Q$, respectively.
Combining this obstruction theory with the Hodge homotopy method
introduced in \cite{FKLS2021} and further developed in
\cite{FiorenzaLe2025}, we give a new proof of a theorem of
Crowley--Nordstr\"om \cite{CN}: if $M$ is a closed
$(r-1)$-connected manifold of dimension $4r-1$ with
$b_r(M)\leq3$, and there exists a class $\varphi\in H^{2r-1}(M;\mathbb R)$
such that multiplication by $\varphi$ induces an isomorphism
$H^r(M;\mathbb R)\longrightarrow H^{3r-1}(M;\mathbb R),\,x\longmapsto\varphi\smile x$,
then $M$ is intrinsically formal. Finally, we prove a borderline extension of a vanishing theorem of Fiorenza--L\^e: if an $(r-1)$-connected Poincar\'e DGCA over $\Q$ of degree $n\leq5r-2$ admits a Hodge homotopy and satisfies $b^r\leq2$, then the operations of its transferred minimal unital cyclic $C_\infty$-algebra vanish in every arity $k\geq4$.
	\end{abstract}

	
	\maketitle
	
	\tableofcontents

	\section{Introduction}
	
	Algebraic models for topological spaces were pioneered by Sullivan, who used  differential graded commutative  algebras  (DGCAs) of differential  forms to define the diffeomorphism type of compact manifolds up to finite ambiguity \cites{Sullivan1975,Sullivan1977}. The Sullivan DGCA   model  was refined  by   Halperin-Stasheff \cite{HS1979}.  The Halperin-Stasheff filtered model,  built on the Tate--Jozefiak resolution  of  the cohomology algebra $H^*(\Aa^*)$ of a DGCA  $\Aa^*$ over a field  $\F$  of characteristic 0, provides a method to construct  Sullivan minimal models for simply connected spaces, allowing a systematic approximation of the rational homotopy type \cite{HS1979}.  These models, alongside the Quillen differential graded Lie algebras (DGLA) \cite{Quillen1969} and the Kadeishvili $C_\infty$-algebra model \cite{Kadeishvili2009}, have provided robust tools for understanding homotopy types. 
	
	  In \cite{SS2012} Schlessinger--Stasheff  showed that  the Halperin-Stasheff  model leads to    a representation  of  the set  of homotopy type $(X, i: \Hh \cong  H^*(X))$ as the quotient  of a conical rational algebraic variety $V$ modulo a  pro-unipotent  group action  $G$  \cite{SS2012}. Here $V$ is the space of  perturbations of   the differential in the Tate--Jozefiak  resolution of $ H^*(X)$  and $G$ is the pro-unipotent pro-algebraic group of automorphisms of the resolution that induce the identity on cohomology. (The pro-unipotency of $G$ follows from the fact that for simply connected spaces, any automorphism inducing the identity on cohomology is the exponential of a degree-0 nilpotent derivation.) In this formulation  ``finding a complete set of invariants for  rational homotopy  types seems  to be about   the same  order of difficulty  as finding  a complete  set of invariants  for the $G$-orbits of a variety $V$... except for special cohomology rings'' \cite{SS2012}.   Recent refinements using unital cyclic $C_\infty$-algebras by Kajiura \cite{Kajiura2007, Kajiura2018} and Hamilton--Lazarev \cite{HL2008} have further specialized these tools for closed oriented manifolds.  
	Significant results include Miller’s 1979 formality result for $(k-1)$-connected  closed manifolds of dimension  less than or equal $4k-2$ \cite{Miller1979}, using Quillen theory, and Crowley--Nordstr\"om's 2021 classification of $(r-1)$-connected manifolds of dimension less than or equal $(5r-3)$ via the Bianchi--Massey tensor, using Sullivan theory \cite{CN}.

	This paper  aims to refine  Kadeishvili   rational homotopy theory  (RHT)  classification  using  minimal $C_\infty$-algebras to define  a complete  set of invariants  of the rational homotopy type  of a topological space  whose  cohomology algebra is of finite dimension.    Kadeishvili's results  \cite{Kadeishvili1980}, \cite{Kadeishvili2009} stated  that   for a DGCA  $\Aa^*$ there  exists   a minimal $C_\infty$-algebra   enhancement  $C_\infty\Aa^*$  of $H^* (\Aa^*)$ which is  homotopy equivalent to  $\Aa^*$. Furthermore, assuming that  $H^* (\Aa^*) = H^* (\Bb^*)$,  then $C_\infty\Aa^*$ and $C_\infty\Bb^*$  are $C_\infty$-isotopic if and only if $\Aa^*$ and  $\Bb^*$ are  weakly equivalent.
	
	We introduce   a stratification   of    isotopy modulo $k$, for all $k \in \N^+$, on the  set of all  minimal  $C_\infty$-algebra enhancements of a finite dimensional graded commutative algebra   (GCA) $H^*$  and show that  two   such  minimal  $C_\infty$-algebra enhancements are $C_\infty$-isotopic if  and only if  they are  isotopic  modulo $k$ for every $k \in \N^+$ (Theorem \ref{thm:iso_from_equiv}).
	As a result, we give     classifications of  real  and rational homotopy types of  simply connected smooth manifolds $M$ with  the same  cohomology  algebra   $H^*(M, \F)$, $\F = \Q, \R$,  in terms   of  a sequence  of  obstruction  classes        defining the   extendability  of  isotopy  modulo  $k$   to isotopy  modulo $(k+1)$ of minimal $C_\infty$-algebra enhancements of   $H^*(M, \F)$.  
	We show that the real and rational homotopy types of a closed
$(r-1)$-connected manifold of dimension
$n\leq\ell(r-1)+2$, where $r\geq2$ and $\ell\geq4$, are determined
uniquely by the cohomology algebra $H^*(M;\F)$ and the isotopy class
modulo $(\ell-2)$ of the corresponding minimal unital cyclic
$C_\infty$-algebra enhancement, for $\F=\R$ and $\F=\Q$,
respectively (Theorem \ref{thm:finite}). Furthermore, we illustrate the technique of Hodge   homotopy, introduced  by Fiorenza--Kawai--L\^e--Schwachh\"ofer   in \cite{FKLS2021}, and developed further  by  Fiorenza--L\^e  in \cite{FiorenzaLe2025}, to give a new proof  of a formality  theorem  for $(r-1)$-connected  closed manifolds $M$ of dimension  $4r-1$  with $b_r \le 3$ admitting  the Lefschetz  property  due to Crowley--Nordstr\"om \cite{CN}. We also   prove a borderline extension of a vanishing theorem of
	Fiorenza--L\^e: if an $(r-1)$-connected Poincar\'e DGCA over $\Q$ of degree
	$n\leq5r-2$ admits a Hodge homotopy and satisfies $b^r\leq2$, then
	the operations of its transferred minimal unital cyclic
	$C_\infty$-algebra vanish in every arity $k\geq4$.
	
	Our paper  is organized as follows.	 In  Section \ref{sec:hodge}, we  recall the concepts  of   a Hodge homotopy $d^-$, the   corresponding harmonic subspace,   and the Hodge decomposition, which is  stronger  than  those considered  by Chuang-Lazarev \cite{ChuangLazarev2008} and   by Kajiura \cite{Kajiura2007}.   In Section \ref{sec:cyclicity}, we show  that   the transferred  minimal $C_\infty$-algebra associated   with  a Hodge homotopy $d^-$ is unital  cyclic (Theorem \ref{thm:cyclic}).  We also consider  several symmetries of the   minimal unital  cyclic $C_\infty$-algebra, transferred using  Hodge  homotopy,  that shall  be   needed  in the   proof  of formality  and almost formality theorems  in later sections. Section \ref{sec:class}  is devoted  to  classification  of real and rational  homotopy type  of  simply connected topological spaces, in particular, closed simply connected manifolds, using   minimal $C_\infty$-algebra and  minimal unital  cyclic $C_\infty$-algebra  enhancements and their isotopy modulo $k$, $k \in \N^+$  (Theorem \ref{thm:iso_from_equiv}),  deformation   theory of    filtered differential  graded  Lie algebras (DGLAs), Harrison cohomology,  and  cyclic  Harrison cohomology. We  give explicit formulas for the first and secondary obstruction classes (Proposition \ref{prop:equiv3}, Lemmas \ref{lem:2obst}). We      describe   the  generalized affine  structure  of the  set of complete  invariants  of  the rational  homotopy types  of  simply-connected  topological spaces  with a  given  cohomological algebra of finite dimension  (Theorem \ref{thm:additivek},  Proposition \ref{prop:obstruction-additivity}). As a consequence, we  give a  proof  of Theorem \ref{thm:finite} on  the  rational homotopy type of a $(r-1)$-connected, $r\ge 2$,  manifold  of dimension less than or equal $\ell(r-1)+2$, $\ell\geq4$.  We  illustrate   our technique  in Subsection  \ref{subs:example}.   In Sections  \ref{sec:cavalcanti}  and \ref{sec:cavalcantiex}, we give   proofs of formality and vanishing theorems (Theorems \ref{thm:CCN}, \ref{thm:m4}),  using  Hodge homotopy and  the  resulting  symmetries  of   the transferred  minimal unital cyclic  $C_\infty$-algebras, proved in Section \ref{sec:cyclicity}.  In the last section \ref{sec:final} we   discuss related  results  and some open   questions. 
	
	The results presented in this paper concerning the case of $(r-1)$-connected manifolds of dimension $n = 4r-1$ with $b_r \le 3$, as well as several foundational technical lemmas in the preliminaries, were developed in collaboration with Domenico Fiorenza. The author is deeply grateful for his remarks  about a relation between  the Hamilton-Lazarev obstruction theory  and the theory  we develop in this paper, and for his generous guidance. The proofs regarding  the higher-dimensional bounds up to $5r-2$ and $\ell(r-1)+2$,  the  introduction  of  isotopy modulo $k$ and the resulting stratification on  the set of minimal $C_\infty$-algebra  enhancements, the gauge-theoretic analysis of the  higher obstruction classes are the work of the author.

		\section*{Acknowledgement}
	
	The author wishes to thank Sapienza Universit\`a di Roma for its
hospitality and excellent working conditions during her visit in
September 2025, when part of this work was discussed. The author
thanks Pavel H\'ajek and Johannes Nordstr\"om for sending her their
papers \cite{CHV2022} and \cite{NN2021}, respectively. The author
acknowledges the assistance of Google's Gemini and OpenAI's ChatGPT in
checking the Fern\'andez--Mu\~noz example in Subsection
\ref{subs:example}, and thanks xAI's Grok for symbolic calculations
related to the Ricci cochain in Proposition \ref{prop:a3}. The author
also acknowledges OpenAI's ChatGPT, Anthropic's Claude, and DeepSeek
for assistance in improving the readability of the exposition. This
research was supported by the Institute of Mathematics of the Czech
Academy of Sciences (RVO 67985840).

	\section{Hodge homotopy and unital cyclic $C_\infty$-algebras}\label{sec:hodge}

In this section we briefly recall the concept of a Poincar\'e DGCA
(PDGCA) admitting a Hodge homotopy, introduced by
Fiorenza--Kawai--L\^e--Schwachh\"ofer in \cite{FKLS2021} and developed
further by Fiorenza--L\^e in \cite{FiorenzaLe2025}. We also observe
that, with the convention of \cite[Definition 3.2(2)]{ChuangLazarev2008},
which allows a possibly degenerate invariant pairing, every PDGCA is
naturally a strictly unital cyclic $C_\infty$-algebra concentrated in
arities $1$ and $2$; see Remark \ref{rem:gss}(3)--(4).
	
	Throughout this section, $\F$ is a field of characteristic zero. Given a vector space $V$ over $\F$, we denote by $V^\vee$ its dual space.
	
	\begin{definition}[\cite{FiorenzaLe2025}, Definition 2.1; cf. \cite{LS04}, Definition 4.1]
		A \emph{PDGCA $(\Aa^*,d)$ of degree $n$} over $\F$ is a DGCA
		\[
		\Aa^*=\bigoplus_{k=0}^n \Aa^k
		\]
		whose cohomology ring $H^*(\Aa)$ is finite-dimensional over $\F$ and is equipped with a \emph{fundamental class}
		\[
		\int\in \bigl(H^n(\Aa)\bigr)^\vee
		\]
		such that the pairing
		\begin{equation}\label{eq:pairingh}
			\langle \alpha^k,\beta^l\rangle=
			\begin{cases}
				\displaystyle \int \alpha^k\cdot\beta^l,& k+l=n,\\
				0,&\text{otherwise},
			\end{cases}
		\end{equation}
		for $\alpha^k\in H^k(\Aa)$ and $\beta^l\in H^l(\Aa)$ is nondegenerate; equivalently,
		\[
		\langle \alpha,H^*(\Aa)\rangle=0
		\quad\Longleftrightarrow\quad
		\alpha=0.
		\]
	\end{definition}

	The pairing on $H^*(\Aa)$ induces a pairing on $\Aa^*$ by
	\begin{equation}\label{eq:pairing}
		\langle \alpha^k,\beta^l\rangle=
		\begin{cases}
			\displaystyle \int[\alpha^k\cdot\beta^l],& k+l=n,\\
			0,&\text{otherwise},
		\end{cases}
	\end{equation}
	where
	\[
	[\,\cdot\,]\colon \Aa^n=\ker\bigl(d\colon \Aa^n\to \Aa^{n+1}\bigr)\longrightarrow H^n(\Aa)
	\]
	is the canonical quotient map. Following \cite[Definition 2.2]{FKLS2021}, the PDGCA $(\Aa^*,d)$ is called \emph{nondegenerate} if the pairing \eqref{eq:pairing} on $\Aa^*$ is nondegenerate. This is the chain-level Poincar\'e duality condition appearing in the notion of an oriented differential Poincar\'e duality algebra; cf. \cite[Definition 2.2]{LS04}. A PDGCA $(H^*,d=0)$ is called a \emph{Poincar\'e GCA}, abbreviated as  PGCA.

\begin{definition}\label{def:hodgehomotopy}
	\cite[Definition 2.2]{FiorenzaLe2025}.
	Let $(\Aa^*,d,\cdot,\langle-,-\rangle)$ be a Poincar\'e DGCA. A \emph{Hodge homotopy} for $(\Aa^*,d,\cdot,\langle-,-\rangle)$ is a degree $-1$ operator
	\[
	d^-\colon \Aa^*\longrightarrow \Aa^{*-1}
	\]
	such that
	\begin{equation}\label{eq:commute}
		(d^-)^2=0,\qquad d^-dd^-=d^-,\qquad dd^-d=d,
	\end{equation}
	and
	\begin{equation}\label{eq:orthogonality}
		\langle \operatorname{Im}(d^-),\operatorname{Im}(d^-)\rangle=0,
		\qquad
		\langle \operatorname{Im}(\pi_{\mathcal H^*}),\operatorname{Im}(d^-)\rangle=0,
	\end{equation}
	where
	\begin{equation}\label{eq:proj}
		\pi_{\mathcal H^*}:=\operatorname{id}_{\Aa^*}-[d,d^-]
		=\operatorname{id}_{\Aa^*}-dd^--d^-d,
	\end{equation}
	and $[d,d^-]=dd^-+d^-d$ is the graded commutator.
\end{definition}

\begin{remark}
	\cite[Remark 2.3]{FiorenzaLe2025}.
	It follows from \eqref{eq:commute} that $dd^-$ and $d^-d$ are idempotents satisfying
	\[
	(dd^-)(d^-d)=(d^-d)(dd^-)=0.
	\]
	Consequently, $dd^-$, $d^-d$, and $\pi_{\mathcal H^*}$ are projection operators.
\end{remark}

	\begin{definition}\label{def:harmonic}
		\cite[Definition 2.2]{FKLS2021} and \cite[Definition 2.4]{FiorenzaLe2025}.
		The graded subspace
		\[
		\mathcal H^*:=\operatorname{Im}(\pi_{\mathcal H^*})
		\]
		is called the \emph{harmonic subspace} associated with the Hodge homotopy $d^-$.
	\end{definition}

	\begin{remark}\label{rem:hodgedecomp}
		\cite[Remark 2.5]{FiorenzaLe2025}.
		It follows from the definition that
		\[
		d\pi_{\mathcal H^*}=\pi_{\mathcal H^*}d
		=d^-\pi_{\mathcal H^*}=\pi_{\mathcal H^*}d^-=0,
		\]
		and that there is a direct sum decomposition
		\begin{equation}\label{eq:hodge1}
			\Aa^*
			=\mathcal H^*\oplus
			\underbrace{dd^-\Aa^*\oplus d^-d\Aa^*}_{\mathcal L_{\mathcal A}^*}.
		\end{equation}
		Moreover, $(\mathcal L_{\mathcal A}^*,d)$ is acyclic and
		\[
		d^-(\mathcal L_{\mathcal A}^*)\subseteq \mathcal L_{\mathcal A}^*.
		\]
		The composition
		\[
		\mathcal H^*\longrightarrow \ker d\longrightarrow H^*(\Aa)
		\]
		is an isomorphism of graded vector spaces.
	\end{remark}

\begin{remark}\label{rem:gss}
	\begin{enumerate}
		\item
		For homogeneous $x,y\in \Aa^*$, the Poincar\'e pairing is graded symmetric:
		\begin{equation}\label{eq:gss}
			\langle x,y\rangle=(-1)^{|x||y|}\langle y,x\rangle.
		\end{equation}
		
		\item
		Since the induced functional on $\Aa^n$ vanishes on exact elements, $d$ is graded skew-adjoint: for homogeneous $x,y\in \Aa^*$,
		\begin{equation}\label{eq:adjoind}
			\langle dx,y\rangle=-(-1)^{|x|}\langle x,dy\rangle.
		\end{equation}
		Together with \eqref{eq:orthogonality}, equations \eqref{eq:gss} and \eqref{eq:adjoind} imply  that the decomposition
		\[
		\Aa^*=\mathcal H^*\oplus \mathcal L_{\mathcal A}^*
		\]
		is orthogonal:
		\[
		\Aa^*=\mathcal H^*\oplus^\perp \mathcal L_{\mathcal A}^*.
		\]
		
		\item
		The pairing is invariant: for homogeneous $x,y,z\in \Aa^*$,
		\begin{equation}\label{eq:wFrob}
			\langle xy,z\rangle=\langle x,yz\rangle.
		\end{equation}
		Set
		\[
		\mu_1(x)=dx,\qquad \mu_2(x,y)=xy,\qquad \mu_k=0\quad(k\ge 3).
		\]
		Then \eqref{eq:adjoind}, \eqref{eq:gss}, and \eqref{eq:wFrob} give
		\[
		\langle \mu_1(x),y\rangle
		=-(-1)^{|x||y|}\langle \mu_1(y),x\rangle
		\]
		and
		\[
		\langle \mu_2(x,y),z\rangle
		=(-1)^{|x|(|y|+|z|)}\langle \mu_2(y,z),x\rangle.
		\]
	 These correspond to the standard cyclicity identities \cite{ChuangLazarev2008}, assuming the convention where the cyclic permutation $z(k)=(0  \,1\,\ldots\,k)\in \mathbb{S}_{k+1}$ 
		acts on the $k$-ary operation by
		\[
		\mu_k^{z(k)}=(-1)^k\mu_k.
		\]
 Here the pairing is allowed to be degenerate; if $(\Aa^*,d)$ is nondegenerate, it is nondegenerate also at the chain level.
		
		\item
		Recall that a $C_\infty$-algebra is \emph{strictly unital} if it has an element $1$ of degree $0$ such that
		\[
		\mu_1(1)=0,
		\qquad
		\mu_2(1,a)=\mu_2(a,1)=a,
		\]
		and, for every $k\ge 3$,
		\[
		\mu_k(a_1,\ldots,1,\ldots,a_k)=0
		\]
		for all $a,a_1,\ldots,a_k\in \Aa^*$; see \cite[(2.8)]{FiorenzaLe2025} and cf. \cite[Definition 8]{CG2008} after translating from the shifted convention used there. Since $\Aa^*$ is a unital DGCA, the cyclic $C_\infty$-structure in (3) is strictly unital. Moreover, $d^-(1)=0$ for degree reasons and $d(1)=0$, hence
		\[
		\pi_{\mathcal H^*}(1)=1.
		\]
		Thus $1\in \mathcal H^0$ is the harmonic representative of the unit $[1]\in H^0(\Aa)$.
	\end{enumerate}
\end{remark}

\begin{examples}\label{ex:hodge}
	\begin{enumerate}
		\item
		Our principal example of a PDGCA over $\R$ admitting a Hodge homotopy is the de Rham complex $(\Om^*(M),d)$ of a closed, connected, oriented Riemannian manifold $M$. Its harmonic subspace is the space of harmonic forms on $M$; see \cite[Remark 2.7(1)]{FKLS2021} and \cite[Example 2.7]{FiorenzaLe2025}. In this case, the pairing \eqref{eq:pairing} is
		\[
		\langle \alpha,\beta\rangle=\int_M \alpha\wedge\beta,
		\]
		and the Hodge homotopy is
		\[
		d^-=Gd^*,
		\]
		where $G$ is the Green operator and $d^*$ is the formal adjoint of $d$. In \cite[Theorem 3.10 and (4.1)]{FKLS2021}, Fiorenza--Kawai--L\^e--Schwachh\"ofer give an explicit construction of a finite-dimensional nondegenerate PDGCA weakly equivalent to $(\Om^*(M),d)$ when $b_1(M)=0$.
		
		\item
		Let $(\Aa^*,d)$ be a simply connected PDGCA over $\Q$. By \cite[Theorem 1.1]{LS08}, it is weakly equivalent to a finite-dimensional nondegenerate PDGCA over $\Q$, equivalently, to a differential Poincar\'e duality algebra. By \cite[Lemma 11.1]{CFL2020}, every such finite-dimensional nondegenerate PDGCA admits a Hodge homotopy.
	\end{enumerate}
\end{examples}

	\section{Cyclicity of the minimal unital $C_\infty$-algebra via Hodge homotopy transfer}\label{sec:cyclicity}
	
Let $(\Aa^*,d,\cdot,\langle-,-\rangle)$ be a PDGCA of degree $n$ over a field $\F$ of characteristic zero, equipped with a Hodge homotopy $d^{-}$. We do not require the pairing $\langle-,-\rangle$ on $\Aa^*$ to be nondegenerate, nor do we assume that $H^1(\Aa^*)=0$. We show that the unital minimal $C_\infty$-enhancement of $(\Hh^*,d=0)$ obtained from $(\Aa^*,d,\cdot,\langle-,-\rangle)$ by homotopy transfer via $d^-$ is cyclic (Theorem \ref{thm:cyclic}).

Let $j\colon\Hh^*\hookrightarrow\Aa^*$ denote the inclusion.

\begin{lemma}\label{lem:adjoint}
	We have:
	\begin{enumerate}
		\item $j$ is an isometric embedding;
		\item $\pi_{\Hh^*}$ is graded self-adjoint:
		\begin{equation}\label{eq:adjoini}
			\langle \pi_{\Hh^*}x,y\rangle
			=
			\langle x,\pi_{\Hh^*}y\rangle
			\qquad
			\text{for all homogeneous }x,y\in\Aa^*.
		\end{equation}
		In particular, for all $x\in\Hh^*$ and $y\in\Aa^*$,
		\[
		\langle x,y\rangle=\langle x,\pi_{\Hh^*}y\rangle.
		\]
		\item $d^-$ is graded self-adjoint:
		\begin{equation}\label{eq:adjoin1}
			\langle d^-x,y\rangle
			=
			(-1)^{|x|}\langle x,d^-y\rangle
			\qquad
			\text{for all homogeneous }x,y\in\Aa^*.
		\end{equation}
	\end{enumerate}
\end{lemma}

\begin{proof}
	The first two assertions follow from the orthogonal Hodge decomposition \eqref{eq:hodge1}. To prove the last assertion, we use
	\[
	\operatorname{id}_{\Aa^*}
	=
	\pi_{\Hh^*}+dd^-+d^-d,
	\]
	the graded skew-adjointness of $d$ in \eqref{eq:adjoind}, and the orthogonality relations:
	\[
	\langle d^-x,y\rangle
	=
	\langle d^-x,dd^-y\rangle
	=
	(-1)^{|x|}\langle dd^-x,d^-y\rangle
	=
	(-1)^{|x|}\langle x,d^-y\rangle.
	\]
\end{proof}

\begin{remark}
	By \eqref{eq:proj}, the maps
	\[
	(\Hh^*,0)
	\mathrel{\substack{\xrightarrow{\ j\ }\\[-0.6ex]\xleftarrow[\ \pi_{\Hh^*}\ ]{}}}
	(\Aa^*,d)
	\]
	satisfy
	\[
	\pi_{\Hh^*}j=\operatorname{id}_{\Hh^*},
	\qquad
	\operatorname{id}_{\Aa^*}-j\pi_{\Hh^*}
	=
	dd^-+d^-d.
	\]
	Hence they form contraction data, with homotopy $d^-$ under the convention
	$\operatorname{id}-j\pi=dh+hd$ (and with homotopy $-d^-$ under the convention
	$j\pi-\operatorname{id}=dh+hd$). We may therefore transfer the unital DGCA structure on $(\Aa^*,d)$ to a $C_\infty$-structure on $\Hh^*$. This structure is minimal because the differential on $\Hh^*$ is zero.
\end{remark}

\begin{theorem}\label{thm:cyclic}
	The transferred minimal $C_\infty$-algebra associated with the Hodge homotopy $d^-$, denoted by $C_\infty(\Hh^*)$, is unital and cyclic. More precisely, under the convention that the cyclic generator
	\[
	z(k)=(0,1,\dots,k)\in\mathbb S_{k+1}
	\]
	acts on the operation $m_k$ by 
	$m_k^{z(k)}=(-1)^km_k$, one has
	\begin{equation}\label{eq:cyclicsign}
		\langle m_k(x_1,\dots,x_k),x_{k+1}\rangle
		=
		(-1)^k(-1)^\epsilon
		\langle m_k(x_2,\dots,x_{k+1}),x_1\rangle,
	\end{equation}
	for all homogeneous $x_1,\dots,x_{k+1}\in\Hh^*$, where
	\[
	\epsilon
	=
	|x_1|\bigl(|x_2|+\cdots+|x_{k+1}|\bigr).
	\]
\end{theorem}

\begin{proof}[Proof of Theorem \ref{thm:cyclic}]
	Since the Poincar\'e pairing is graded symmetric by \eqref{eq:gss}, Lemma \ref{lem:adjoint} and the defining properties of a Hodge homotopy show that the pairing $\langle-,-\rangle$, the projection $\pi_{\Hh^*}$, and the homotopy $d^-$ form a harmonious Hodge decomposition in the sense of \cite[Definition 2.1]{ChuangLazarev2008}. Hence \cite[Theorem 4.2]{ChuangLazarev2008} implies that the transferred minimal $C_\infty$-algebra is cyclic.
Moreover, the DGCA structure on $\Aa^*$ is a unital $C_\infty$-algebra with unit $1\in\Aa^0$. The Hodge homotopy satisfies the side conditions
	\begin{equation}\label{eq:side}
		\pi_{\Hh^*}d^-=d^-\pi_{\Hh^*}=0,
		\qquad
		(d^-)^2=0,
		\qquad
		d^-(1)=0.
	\end{equation}
	Therefore, by \cite[Theorems 10 and 12]{CG2008}, the transferred minimal $C_\infty$-algebra $C_\infty(\Hh^*)$ is unital (under the translation of  conventions noted in Remark \ref{rem:gss}(4)).
\end{proof}

\begin{remark}\label{rem:cyclic}
	The cyclicity assertion in Theorem \ref{thm:cyclic} can also be deduced from \cite[Lemma 5.23]{Kajiura2007}, where Kajiura uses a Hodge--Kodaira decomposition to obtain a minimal cyclic $A_\infty$-algebra from a cyclic $A_\infty$-algebra. Our Hodge decomposition \eqref{eq:hodge1} is a particular case of the Hodge--Kodaira decomposition considered in \cite[Section 5.2]{Kajiura2007}.
\end{remark}

\begin{remark}\label{rem:vanishmkr}
	Let $r\geq1$, and let $(\Aa^*,d,\cdot,\langle-,-\rangle)$ be an $(r-1)$-connected Poincar\'e DGCA of degree $n$ admitting a Hodge homotopy $d^-$ with harmonic subspace $\Hh^*$. Explicit formulas for the operations $m_k$ of $C_\infty(\Hh^*)$ are obtained from \cite[Theorem 3.4]{Merkulov1999} and \cite[Theorem 12]{CG2008}. Define recursively
	\begin{align}\label{eq:hat-mk}
		\widehat m_2(\alpha_1,\alpha_2)
		&=
		\alpha_1\cdot\alpha_2,
		\notag\\
		\widehat m_k(\alpha_1,\dots,\alpha_k)
		&=
		(-1)^{k-1}d^-\widehat m_{k-1}(\alpha_1,\dots,\alpha_{k-1})\cdot\alpha_k
		\notag\\
		&\quad
		-
		(-1)^{k|\alpha_1|}
		\alpha_1\cdot d^-\widehat m_{k-1}(\alpha_2,\dots,\alpha_k)
		\notag\\
		&\quad
		-
		\sum_{i=2}^{k-2}
		(-1)^\nu
		d^-\widehat m_i(\alpha_1,\dots,\alpha_i)
		\cdot
		d^-\widehat m_{k-i}(\alpha_{i+1},\dots,\alpha_k),
	\end{align}
	for $k\geq3$, where
	\[
	\nu
	=
	i+(k-i-1)\bigl(|\alpha_1|+\cdots+|\alpha_i|\bigr),
	\]
	the $\alpha_i$ are elements of $\Hh^*\subseteq\Aa^*$, and the sum is empty for $k=3$. Then
	\begin{equation}\label{eq:mk}
		m_k(\alpha_1,\dots,\alpha_k)
		=
		\pi_{\Hh^*}\bigl(\widehat m_k(\alpha_1,\dots,\alpha_k)\bigr).
	\end{equation}
	For instance,
	\begin{align}
		m_3(x,y,z)
		&=
		\pi_{\Hh^*}\bigl(d^-(x\cdot y)\cdot z\bigr)
		\notag\\
		&\quad
		-
		(-1)^{|x|}
		\pi_{\Hh^*}\bigl(x\cdot d^-(y\cdot z)\bigr).
		\label{eq:m3}
	\end{align}
	The transferred structure is strictly unital:
	\[
	m_2(1,\alpha)=m_2(\alpha,1)=\alpha,
	\]
	and, for every $k\geq3$, $m_k(\alpha_1,\dots,\alpha_k)=0$ whenever one of the $\alpha_i$ is equal to $1$.
	
	In \cite[Remark 2.15]{FiorenzaLe2025}, Fiorenza--L\^e showed that, for every $k\geq3$ and every $\alpha\in\Hh^r$,
	\begin{equation}\label{eq:vanishmkr}
		\widehat m_k(\alpha,\dots,\alpha)
		=
		m_k(\alpha,\dots,\alpha)
		=
		0.
	\end{equation}
\end{remark}

\begin{lemma}[\bfseries Generalized Odd Symmetry]\label{lem:oddsymmetry}
	Assume that $\Aa^*$ is $(r-1)$-connected and that $r$ is odd. Then, for every $k\geq3$ and every $x,y,z\in\Hh^r$,
	\begin{align}
		\widehat m_k(x,\dots,x,y)
		&=
		(-1)^{k-1}\widehat m_k(y,x,\dots,x),
		\label{eq:symmetry1}\\
		\widehat m_k(x,y,\dots,y,z)
		&=
		(-1)^{k-1}\widehat m_k(z,y,\dots,y,x).
		\label{eq:symmetry2}
	\end{align}
\end{lemma}

\begin{proof}
	We prove the following more general reflection identity:
	\begin{equation}\label{eq:odd-reflection}
		\widehat m_k(a_1,\dots,a_k)
		=
		(-1)^{k-1}\widehat m_k(a_k,\dots,a_1)
	\end{equation}
	for every $k\geq2$ and every $a_1,\dots,a_k\in\Hh^r$.
	
	For convenience, set
	\[
	D_j(a_1,\dots,a_j)
	\coloneqq
	d^-\widehat m_j(a_1,\dots,a_j).
	\]
	Since $\widehat m_j$ has degree $2-j$ and each $a_s$ has the odd degree $r$, we have
	\[
	\bigl|\widehat m_j(a_1,\dots,a_j)\bigr|
	=
	jr+2-j
	\equiv0\pmod2.
	\]
	Consequently,
	\[
	\bigl|D_j(a_1,\dots,a_j)\bigr|
	\equiv1\pmod2.
	\]
	Thus every $a_s$ and every value of $D_j$ occurring below has odd degree.
	
	For inputs $a_1,\dots,a_k\in\Hh^r$, the sign in the quadratic part of
	\eqref{eq:hat-mk} simplifies to
	\[
	\nu
	=
	i+(k-i-1)(|a_1|+\cdots+|a_i|)
	\equiv
	i+(k-i-1)ir
	\equiv
	i(k-i)
	\pmod2,
	\]
	because $r$ is odd. Moreover,
	\[
	-(-1)^{k|a_1|}
	=
	-(-1)^{kr}
	=
	(-1)^{k-1}.
	\]
	Hence the recursion \eqref{eq:hat-mk} becomes
	\begin{align}
		\widehat m_k(a_1,\dots,a_k)
		&=
		(-1)^{k-1}
		D_{k-1}(a_1,\dots,a_{k-1})\cdot a_k
		\notag\\
		&\quad+
		(-1)^{k-1}
		a_1\cdot D_{k-1}(a_2,\dots,a_k)
		\notag\\
		&\quad-
		\sum_{i=2}^{k-2}
		(-1)^{i(k-i)}
		D_i(a_1,\dots,a_i)
		\cdot
		D_{k-i}(a_{i+1},\dots,a_k).
		\label{eq:odd-recursion}
	\end{align}
	
	We now prove \eqref{eq:odd-reflection} by induction on $k$.
	
	For $k=2$, graded commutativity and the oddness of $r$ give
	\[
	\widehat m_2(a_1,a_2)
	=
	a_1\cdot a_2
	=
	(-1)^{r^2}a_2\cdot a_1
	=
	-\widehat m_2(a_2,a_1),
	\]
	which is precisely \eqref{eq:odd-reflection} for $k=2$.
	
	Assume that \eqref{eq:odd-reflection} holds in every arity smaller
	than $k$. Applying $d^-$ to the inductive identity gives
	\begin{equation}\label{eq:D-reflection}
		D_j(a_1,\dots,a_j)
		=
		(-1)^{j-1}D_j(a_j,\dots,a_1)
	\end{equation}
	for every $2\leq j<k$.
	
	We compare the right-hand side of \eqref{eq:odd-recursion} with the
	same formula applied to the reversed sequence
	\[
	(a_k,a_{k-1},\dots,a_1).
	\]
	
	First consider the two boundary terms. By \eqref{eq:D-reflection},
	\begin{align*}
		&D_{k-1}(a_k,\dots,a_2)\cdot a_1
		+
		a_k\cdot D_{k-1}(a_{k-1},\dots,a_1)
		\\
		&\qquad=
		(-1)^{k-2}
		D_{k-1}(a_2,\dots,a_k)\cdot a_1
		\\
		&\qquad\quad+
		(-1)^{k-2}
		a_k\cdot D_{k-1}(a_1,\dots,a_{k-1}).
	\end{align*}
	Both factors in each product are odd. Graded commutativity therefore
	gives
	\[
	D_{k-1}(a_2,\dots,a_k)\cdot a_1
	=
	-a_1\cdot D_{k-1}(a_2,\dots,a_k)
	\]
	and
	\[
	a_k\cdot D_{k-1}(a_1,\dots,a_{k-1})
	=
	-D_{k-1}(a_1,\dots,a_{k-1})\cdot a_k.
	\]
	Consequently,
	\begin{align}
		&D_{k-1}(a_k,\dots,a_2)\cdot a_1
		+
		a_k\cdot D_{k-1}(a_{k-1},\dots,a_1)
		\notag\\
		&\qquad=
		(-1)^{k-1}
		\bigl(
		D_{k-1}(a_1,\dots,a_{k-1})\cdot a_k
		+
		a_1\cdot D_{k-1}(a_2,\dots,a_k)
		\bigr).
		\label{eq:odd-boundary-reflection}
	\end{align}
	
	Next consider a summand in the quadratic term. After reversing the
	inputs, reindex the sum by replacing $i$ with $k-i$. The sign
	coefficient is unchanged, since
	\[
	(k-i)i=i(k-i).
	\]
	Using \eqref{eq:D-reflection} for the two factors, we obtain
	\begin{align*}
		&D_{k-i}(a_k,\dots,a_{i+1})
		\cdot
		D_i(a_i,\dots,a_1)
		\\
		&\qquad=
		(-1)^{k-i-1}(-1)^{i-1}
		D_{k-i}(a_{i+1},\dots,a_k)
		\cdot
		D_i(a_1,\dots,a_i)
		\\
		&\qquad=
		(-1)^{k-2}
		D_{k-i}(a_{i+1},\dots,a_k)
		\cdot
		D_i(a_1,\dots,a_i).
	\end{align*}
	The two displayed factors are odd, so interchanging them contributes
	one additional minus sign. Therefore
	\begin{align}
		&D_{k-i}(a_k,\dots,a_{i+1})
		\cdot
		D_i(a_i,\dots,a_1)
		\notag\\
		&\qquad=
		(-1)^{k-1}
		D_i(a_1,\dots,a_i)
		\cdot
		D_{k-i}(a_{i+1},\dots,a_k).
		\label{eq:odd-quadratic-reflection}
	\end{align}
	
	Equations \eqref{eq:odd-boundary-reflection} and
	\eqref{eq:odd-quadratic-reflection}, together with the recursion
	\eqref{eq:odd-recursion}, show that
	\[
	\widehat m_k(a_k,\dots,a_1)
	=
	(-1)^{k-1}
	\widehat m_k(a_1,\dots,a_k).
	\]
	Since $((-1)^{k-1})^2=1$, this is equivalent to
	\eqref{eq:odd-reflection}. The induction is complete.
	
	Finally, applying \eqref{eq:odd-reflection} to
	\[
	(a_1,\dots,a_k)=(x,\dots,x,y)
	\]
	gives \eqref{eq:symmetry1}, while applying it to
	\[
	(a_1,\dots,a_k)=(x,y,\dots,y,z)
	\]
	gives \eqref{eq:symmetry2}.
\end{proof}

\begin{lemma}[\bfseries Generalized Even Symmetry for $k=3$]\label{lem:evensymmetry3}
	Assume that $\Aa^*$ is $(r-1)$-connected and that $r$ is even. Then, for every $y\in\Hh^r$ and every homogeneous $x,z\in\Hh^*$,
	\begin{align}
		\widehat m_3(x,x,y)
		&=
		-(-1)^{|x|}\widehat m_3(y,x,x),
		\label{eq:symmetry1e}\\
		\widehat m_3(x,y,z)
		&=
		-(-1)^{|x||z|}\widehat m_3(z,y,x).
		\label{eq:symmetry2e}
	\end{align}
	In particular, for every $x,y\in\Hh^r$,
	\begin{equation}\label{eq:xyx-vanish}
		\widehat m_3(x,y,x)=0.
	\end{equation}
\end{lemma}

\begin{proof}
	The identities follow directly from \eqref{eq:hat-mk}, graded commutativity, and the fact that $|y|=r$ is even  (using also that $x^2=0$ whenever $|x|$ is odd). Taking $z=x\in\Hh^r$ in \eqref{eq:symmetry2e} gives
	\[
	\widehat m_3(x,y,x)=-\widehat m_3(x,y,x),
	\]
	and hence \eqref{eq:xyx-vanish}.
\end{proof}

\begin{lemma}[\bfseries Generalized  Even Symmetry for $k=4$]\label{lem:evensymmetry4}       Assume that   $\Aa^*$ is $(r-1)$-connected, with $r$ even. Then for every  $x, y,z \in \Hh ^r$ we have
	\begin{align}  
		\widehat m_4  (x, x, x, y)  &=   -  \widehat m_4  (y, x, x, x) \label{eq:symmetry1e4} \\
		\widehat  m_4  (x, y, y, z) &  =  - \widehat m_4  (z, y, y, x)  \label{eq:symmetry2e4}
	\end{align}	
\end{lemma}

\begin{proof} 
	By \eqref{eq:hat-mk} and by the vanishing condition \eqref{eq:vanishmkr}, we have
	\begin{align*}
		\widehat m_4 (x, x, x, y) &= - d^{-}\widehat m_{3}(x,x, x)\cdot y - x \cdot d^{-} \widehat m_{3} (x, x, y)\\
		&\qquad\qquad -  d^-\widehat m_2 (x, x)\cdot d^{-}\widehat m_{2} (x, y)\\
		&= 
		- x \cdot d^{-} \widehat m_{3} (x, x, y) -  d^-\widehat m_2 (x, x)\cdot d^{-}\widehat m_{2} (x, y).
	\end{align*}
	Similarly, and taking into account \eqref{eq:symmetry1e},
	\begin{align*}
		\widehat m_4 (y, x, x, x) &= -d^{-}\widehat m_{3}(y, x, x)\cdot x - y \cdot d^{-} \widehat m_{3} (x, x, x)\\
		&\qquad\qquad - d^-\widehat m_2 (y, x)\cdot d^{-}\widehat m_{2} (x, x)\\
		&= -d^{-}\widehat m_{3}(y, x, x)\cdot x - d^-\widehat m_2 (y, x)\cdot d^{-}\widehat m_{2} (x, x)\\
		&= d^{-}\widehat m_{3}(x, x,y)\cdot x + d^-\widehat m_2 (x, x)\cdot d^{-}\widehat m_{2} (y, x).
	\end{align*}
	This proves \eqref{eq:symmetry1e4}. The proof of  \eqref{eq:symmetry2e4} is analogous.
\end{proof}

\begin{remark}
	The canonical isomorphism of graded vector spaces
	\[
	\Hh^*\xrightarrow{\sim}H^*(\Aa),\qquad
	\alpha\longmapsto[\alpha],
	\]
	transports the minimal unital cyclic $C_\infty$-algebra structure on the harmonic subspace $\Hh^*$ to a minimal unital cyclic $C_\infty$-algebra structure on $H^*(\Aa)$. By abuse of notation, we denote the transported operations again by $m_k$. Under this identification, $m_2$ is the ordinary multiplication in cohomology.
\end{remark}

	\section{Classification of rational and real homotopy types of closed simply-connected smooth manifolds}\label{sec:class}
	
	In this section, we describe minimal $C_\infty$-algebra enhancements and their isotopies using the completed Harrison differential graded Lie algebra. For $k \in \mathbb{N}^+$, we introduce the concept of \emph{isotopy modulo $k$} (Definition \ref{def:equik}, Corollary \ref{cor:filtered-isotopy}) for minimal $C_\infty$-algebra enhancements of a finite-dimensional GCA $H^*$ over a field $\mathbb{F}$ of characteristic zero. (In this section, we   always  assume that $H^*$ is finite dimensional, unless otherwise stated.) We show that, for any $k\in \mathbb{N}^+$, isotopy modulo $k$ defines an equivalence relation on the set $C_{\infty} (H^*)$ of all minimal $C_\infty$-algebra enhancements $(H^*, m_1=0, m_2, m_3, \ldots)$ of $H^*$ whose multiplication is $m_2$ (Proposition \ref{prop:equik}).  Denoting such a minimal $C_\infty$-algebra enhancement as $m = (m_3, m_4, \ldots)$, we show that two minimal $C_\infty$-algebra enhancements are $C_\infty$-isotopic if and only if they are isotopic modulo $k$ for all $k \in \mathbb{N}^+$ (Theorem \ref{thm:iso_from_equiv}). We show that if $H^*$ is a PGCA, the primary obstruction for isotopy modulo 3 of a minimal unital cyclic $C_\infty$-algebra enhancement $m$ of $H^*$ is the cyclic cohomology class $[Tm_3]\in HC_{\mathrm{Harr}}^{4,-1}(H^*)$ (Proposition \ref{prop:equiv3}). Then we describe the structure of a complete set of invariants of the set $C_\infty(H^*)$ modulo isotopy (Theorem \ref{thm:additivek}, Proposition \ref{prop:obstruction-additivity}). We prove that the real and rational homotopy types of a closed $(r-1)$-connected manifold of dimension $n\leq \ell(r-1)+2$ (where $r\geq2$ and $\ell\geq4$) are determined uniquely by the cohomology algebra $H^*(M;\F)$ and the isotopy class modulo $(\ell-2)$ of the corresponding minimal unital cyclic $C_\infty$-algebra enhancement, for $\F=\R$ and $\F=\Q$, respectively (Theorem \ref{thm:finite}). Finally, we illustrate the secondary obstruction on the simply connected non-formal symplectic $8$-manifold of Fern\'andez and Mu\~noz in Subsection \ref{subs:example}.

\subsection{$C_\infty$-isotopies, filtered differential graded Lie algebras,  isotopies modulo $k$, and  cyclic  Harrison cohomology}
\label{subs:twisting}

\subsubsection{$C_\infty$-isotopies and  filtered DGLAs}\label{ss:fdgla}

	Let us recall   that   a  $C_\infty$-morphism  $\phi = (\phi_1, \phi_2, \phi_3, \ldots)$   from  a minimal $ C_\infty$-algebra  $A =   (H^*, m_1 =0, m_2,\ldots)$  to  a minimal $C_\infty$-algebra $A' =  (H^*, m_1' =0, m_2 ', \ldots)$   must satisfy  \cite[Proposition 10.2.6, p. 374]{LV}: 
\begin{align}
	&\sum_{k =1}^n \sum_{r_1+ \dots + r_k = n} (-1)^\eta m_k' (\phi_{r_1} (a_1, \dots, a_{r_1}), \dots, \phi_{r_k}(a_{n-r_k +1}, \dots, a_n)) \nonumber\\
&\quad= \sum_{k =1}^n \sum_{\lambda=0}^{n-k} (-1)^\xi \phi_{n-k+1} (a_1, \dots, a_\lambda, m_k (a_{\lambda+1}, \dots, a_{\lambda+k}), a_{\lambda+k +1}, \dots, a_n )\label{eq:amorphism}
\end{align} 

	 for   all $n \ge 2$  and
	\[\eta = \sum_{1\le \alpha < \beta \le k} (|a_{{r_1} + \ldots + r_{\alpha -1} +1}| + \ldots + |a_{{r_1} + \ldots + {r_\alpha}}|+ r_\alpha)(1+r_\beta), \]
	\[\xi = k + k\lambda + k ( |a_1| +\ldots + |a_\lambda|) +n. \]
	
We  call  $A$ and $A'$ \emph{$C_\infty$-isotopic}  if  $\phi = (\phi_1 = \Id, \phi_2, \phi_3, \ldots)$, i.e., $\phi$ is   a \emph{$C_\infty$-isotopy}.

   \noindent
 We now describe minimal $C_\infty$-algebra enhancements and their
 $C_\infty$-isotopies using the completed Harrison differential graded Lie algebra.
 This is equivalent to the standard description in terms of coderivations
 of the cofree conilpotent coalgebra governing $C_\infty$-algebras;
 see, for example, \cite{LV,HL2009}.
 
 Let
 \[
 \mathfrak g_{H^*}
 \coloneqq
 \prod_{n\geq2}C_{\mathrm{Harr}}^{n,*}(H^*,H^*)
 \]
 and assign to
 $f\in C_{\mathrm{Harr}}^{n,j}(H^*,H^*)$ the total degree
 \[
 |f|_{\mathfrak g}\coloneqq n+j-1.
 \]
 The Gerstenhaber bracket is
 \begin{equation}\label{eq:gbracket}
 	[f,g]
 	\coloneqq
 	f\circ g-
 	(-1)^{|f|_{\mathfrak g}|g|_{\mathfrak g}}g\circ f.
 \end{equation}
 The multiplication $m_2$ determines the differential
 \[
 \delta f\coloneqq[m_2,f].
 \]
 Since $[m_2,m_2]=0$, we have $\delta^2=0$, and hence
 \[
 \mathfrak g_{H^*}
 =
 \left(
 \prod_{n\geq2}C_{\mathrm{Harr}}^{n,*}(H^*,H^*),
 \delta,[-,-]
 \right)
 \]
 is a differential graded Lie algebra.
 
 Let
 \[
 m=m_3+m_4+\cdots,
 \qquad
 m_n\in C_{\mathrm{Harr}}^{n,2-n}(H^*,H^*).
 \]
 Thus $|m_n|_{\mathfrak g}=1$. The operations
 $(m_1=0,m_2,m_3,\ldots)$ define a minimal $C_\infty$-algebra
 enhancement of $H^*$ if and only if the total coderivation $m_2+m$
 squares to zero. Equivalently, $m$ satisfies the Maurer--Cartan equation
 \begin{equation}\label{eq:GMC}
 	\delta m+\frac12[m,m]=0.
 \end{equation}
 This is the differential graded Lie algebra formulation of
 Kadeishvili's twisting-element description
 \cite{Kadeishvili1988,Kadeishvili2011}.
 
 We use the shifted arity filtration
 \begin{equation}\label{eq:HL}
 	\mathcal F^p\mathfrak g_{H^*}
 	\coloneqq
 	\prod_{n\geq p+1}C_{\mathrm{Harr}}^{n,*}(H^*,H^*),
 	\qquad p\geq1.
 \end{equation}
 Then
 \[
 \delta\bigl(\mathcal F^p\mathfrak g_{H^*}\bigr)
 \subseteq
 \mathcal F^{p+1}\mathfrak g_{H^*}
 \subseteq
 \mathcal F^p\mathfrak g_{H^*},
 \]
 \[
 [\mathcal F^p\mathfrak g_{H^*},
 \mathcal F^q\mathfrak g_{H^*}]
 \subseteq
 \mathcal F^{p+q}\mathfrak g_{H^*},
 \]
 and
 \[
 \mathfrak g_{H^*}
 \cong
 \varprojlim_p
 \mathfrak g_{H^*}/\mathcal F^p\mathfrak g_{H^*}.
 \]
 Thus $\mathfrak g_{H^*}$ is complete and pronilpotent with respect
 to the shifted arity filtration.
 
 The associated prounipotent gauge group is
 \[
 \exp\bigl(\mathcal F^1\mathfrak g_{H^*}^0\bigr).
 \]
 For $p\in\mathcal F^1\mathfrak g_{H^*}^0$, its action on
 Maurer--Cartan elements is
 \begin{equation}\label{eq:gaugeaction}
 	e^{\operatorname{ad}_p}*m
 	\coloneqq
 	e^{\operatorname{ad}_p}(m)
 	-
 	\frac{e^{\operatorname{ad}_p}-1}
 	{\operatorname{ad}_p}
 	(\delta p),
 \end{equation}
 where
 \[
 \frac{e^{\operatorname{ad}_p}-1}
 {\operatorname{ad}_p}
 =
 \sum_{j\geq0}
 \frac{(\operatorname{ad}_p)^j}{(j+1)!}.
 \]
 This is the standard gauge action on Maurer--Cartan elements of a
 complete pronilpotent DGLA; see, for example,
 \cite[Lemma 2.8]{GM1988}.
 
 \begin{remark}\label{rem:log}
 	Under the completed coderivation description, a $C_\infty$-isotopy
 	\[
 	\Phi=(\Phi_1=\Id,\Phi_2,\Phi_3,\ldots)
 	\]
 	is represented by a prounipotent pointed formal coalgebra automorphism.
 	Since the ground field has characteristic zero, it has a unique
 	logarithm
 	\[
 	p=\log\Phi\in\mathcal F^1\mathfrak g_{H^*}^0,
 	\qquad
 	\Phi=\exp(p).
 	\]
 	If the transformed total coderivation is defined by
 	\[
 	m_2+m'
 	=
 	\Phi(m_2+m)\Phi^{-1},
 	\]
 	then
 	\[
 	m'=e^{\operatorname{ad}_p}*m.
 	\]
 	
 	Consequently, two
 	minimal $C_\infty$-algebra enhancements of $H^*$ are $C_\infty$-isotopic if and
 	only if their corresponding Maurer--Cartan elements are gauge
 	equivalent in the complete DGLA $\mathfrak g_{H^*}$.
 \end{remark}

 \begin{lemma}[Affine formula for the gauge action]\label{lem:distMC}
 	Let $\mathfrak g$ be a complete filtered DGLA, let
 	$q\in\mathfrak g^0$ have positive filtration, and let
 	$x,\tau\in\mathfrak g^1$. Then
 	\begin{equation}\label{eq:distMC}
 		e^{\operatorname{ad}_q}*(x+\tau)
 		=
 		\bigl(e^{\operatorname{ad}_q}*x\bigr)
 		+
 		e^{\operatorname{ad}_q}(\tau).
 	\end{equation}
 \end{lemma}
 
 \begin{proof}
 	By \eqref{eq:gaugeaction},
 	\begin{align*}
 		e^{\operatorname{ad}_q}*(x+\tau)
 		&=
 		e^{\operatorname{ad}_q}(x+\tau)
 		-
 		\frac{e^{\operatorname{ad}_q}-1}
 		{\operatorname{ad}_q}
 		(\delta q)\\
 		&=
 		\left(
 		e^{\operatorname{ad}_q}(x)
 		-
 		\frac{e^{\operatorname{ad}_q}-1}
 		{\operatorname{ad}_q}
 		(\delta q)
 		\right)
 		+
 		e^{\operatorname{ad}_q}(\tau)\\
 		&=
 		\bigl(e^{\operatorname{ad}_q}*x\bigr)
 		+
 		e^{\operatorname{ad}_q}(\tau).
 	\end{align*}
 \end{proof}
 \subsubsection{Isotopies modulo $k$}
 
 \begin{definition}[Equality modulo $k$]\label{def:equamodk}
 	Let
 	\[
 	A_1=(H^*,\{m_i^{(1)}\}),
 	\qquad
 	A_2=(H^*,\{m_i^{(2)}\})
 	\]
 	be two minimal $C_\infty$-algebra enhancements of a  GCA
 	$H^*$. We say that $A_1$ and $A_2$ are \emph{equal modulo $k$},
 	and write
 	\[
 	A_1=A_2\pmod{k},
 	\]
 	if
 	\[
 	m_i^{(1)}=m_i^{(2)}
 	\qquad
 	\text{for every }2\leq i\leq k.
 	\]
 \end{definition}
 
 \begin{definition}[Isotopy modulo $k$]\label{def:equik}
 	Let $A_1$ and $A_2$ be minimal $C_\infty$-algebra enhancements of
 	$H^*$. We say that $A_1$ is \emph{isotopic to $A_2$ modulo $k$},
 	and write
 	\[
 	A_1\sim_k A_2,
 	\]
 	if there exists a $C_\infty$-isotopy
 	\[
 	\phi=(\phi_1=\Id,\phi_2,\phi_3,\ldots)
 	\]
 	such that
 	\[
 	\phi(A_1)=A_2\pmod{k}.
 	\]
 \end{definition}
 
 \begin{lemma}\label{lem:truncation-equivariance}
 	Let $A=(H^*,\{m_i\})$ and $A'=(H^*,\{m_i'\})$ be two minimal
 	$C_\infty$-algebra enhancements of the same GCA $H^*$. Let
 	\[
 	\Phi=(\Id,\Phi_2,\Phi_3,\ldots)
 	\]
 	be the pointed formal map underlying a $C_\infty$-isotopy, and let
 	$\Phi(A)$ and $\Phi(A')$ denote the transported structures. If
 	\[
 	A=A'\pmod{k},
 	\]
 	then
 	\[
 	\Phi(A)=\Phi(A')\pmod{k}.
 	\]
 \end{lemma}
 
 \begin{proof}
 	For every $n$, the arity-$n$ component of the transported structure
 	$\Phi(A)$ depends only on the operations
 	\[
 	m_2,\ldots,m_n
 	\]
 	and on the components of $\Phi$ of arity at most $n$. Hence, if
 	$m_i=m_i'$ for every $i\leq k$, then the transported operations of
 	$\Phi(A)$ and $\Phi(A')$ agree in every arity at most $k$.
 \end{proof}
 
 \begin{proposition}\label{prop:equik}
 	For every integer $k\geq2$, isotopy modulo $k$ is an equivalence
 	relation on the set of minimal $C_\infty$-algebra enhancements of
 	$H^*$.
 \end{proposition}
 
 \begin{proof}
 	Reflexivity follows by taking the identity $C_\infty$-isotopy
 	\[
 	\Id=(\Id,0,0,\ldots).
 	\]
 	
 	Suppose that $A_1\sim_k A_2$. Then there exists a $C_\infty$-isotopy
 	$\phi$ such that
 	\[
 	\phi(A_1)=A_2\pmod{k}.
 	\]
 	The inverse $\phi^{-1}$ is again a $C_\infty$-isotopy. Applying
 	Lemma~\ref{lem:truncation-equivariance} gives
 	\[
 	A_1
 	=
 	\phi^{-1}\bigl(\phi(A_1)\bigr)
 	=
 	\phi^{-1}(A_2)
 	\pmod{k}.
 	\]
 	Hence $A_2\sim_k A_1$.
 	
 	Now suppose that $A_1\sim_k A_2$ and $A_2\sim_k A_3$. Choose
 	$C_\infty$-isotopies $\phi$ and $\psi$ such that
 	\[
 	\phi(A_1)=A_2\pmod{k},
 	\qquad
 	\psi(A_2)=A_3\pmod{k}.
 	\]
 	By Lemma~\ref{lem:truncation-equivariance},
 	\[
 	\psi\bigl(\phi(A_1)\bigr)
 	=
 	\psi(A_2)
 	\pmod{k}.
 	\]
 	It follows that
 	\[
 	(\psi\circ\phi)(A_1)=A_3\pmod{k}.
 	\]
 	Since $\psi\circ\phi$ is a $C_\infty$-isotopy, we conclude that
 	$A_1\sim_k A_3$.
 \end{proof}
 
 \begin{corollary}[Filtered DGLA interpretation]
 	\label{cor:filtered-isotopy}
 	Let
 	\[
 	m=m_3+m_4+\cdots,
 	\qquad
 	m'=m_3'+m_4'+\cdots
 	\]
 	be the Maurer--Cartan elements corresponding to two minimal
 	$C_\infty$-algebra enhancements $A$ and $A'$ of $H^*$. Then
 	\[
 	A=A'\pmod{k}
 	\quad\Longleftrightarrow\quad
 	m-m'\in\mathcal F^k\mathfrak g_{H^*}.
 	\]
 	Moreover,
 	\[
 	A\sim_k A'
 	\]
 	if and only if there exists
 	\[
 	p\in\mathcal F^1\mathfrak g_{H^*}^0
 	\]
 	such that
 	\[
 	e^{\operatorname{ad}_p}*m
 	\equiv
 	m'
 	\pmod{\mathcal F^k\mathfrak g_{H^*}}.
 	\]
 	Equivalently, isotopy modulo $k$ is the orbit relation induced by
 	the prounipotent gauge group on the Maurer--Cartan elements of the
 	nilpotent quotient
 	\[
 	\mathfrak g_{H^*}/\mathcal F^k\mathfrak g_{H^*}.
 	\]
 \end{corollary}
 
 \begin{proof}
 	The first assertion follows directly from the definition of the
 	shifted arity filtration. For the second assertion, let $\Phi$ be
 	the pointed formal map underlying a $C_\infty$-isotopy and write
 	$\Phi=\exp(p)$, where $p=\log\Phi$. By
 	Remark~\ref{rem:log}, the transported structure $\Phi(A)$ corresponds
 	to the gauge transform $e^{\operatorname{ad}_p}*m$. Therefore
 	\[
 	\Phi(A)=A'\pmod{k}
 	\]
 	if and only if
 	\[
 	e^{\operatorname{ad}_p}*m-m'
 	\in
 	\mathcal F^k\mathfrak g_{H^*}.
 	\]
 \end{proof}

\subsubsection{Cyclic Harrison cohomology}\label{ss:cyc} 
Let $H^*$ be a Poincar\'e GCA. Denote by
\[
C_{\mathrm{Harr,cyc}}^{k,m}(H^*,H^*)
\subseteq
C_{\mathrm{Harr}}^{k,m}(H^*,H^*)
\]
the subspace of Harrison cochains satisfying the cyclicity condition
\eqref{eq:cyclicsign}. For
$f\in C_{\mathrm{Harr,cyc}}^{k,m}(H^*,H^*)$, define the associated
scalar-valued $(k+1)$-form by
\begin{equation}\label{eq:transform}
	Tf(x_1,\ldots,x_k,x_{k+1})
	\coloneqq
	\langle f(x_1,\ldots,x_k),x_{k+1}\rangle .
\end{equation}
Via the Poincar\'e pairing, this identifies the complex of cyclic
$H^*$-valued Harrison cochains with the cyclic Harrison complex.
We denote its cohomology by
\[
HC_{\mathrm{Harr}}^{*,*}(H^*).
\]
Here the second grading is understood with the degree convention
induced from the corresponding $H^*$-valued cochain; equivalently,
one must insert the degree shift determined by the Poincar\'e
pairing when using the scalar-valued convention.

Hamilton--Lazarev prove that the natural map from cyclic Harrison
cohomology to Harrison cohomology with dual coefficients is an
isomorphism in Harrison arity at least $3$; see
\cite[Theorem 3.8]{HL2008} and
\cite[Corollary 8.12]{HL2009}. Consequently, after identifying
$(H^*)^\vee$ with the appropriate degree shift of $H^*$ by the
Poincar\'e pairing, we obtain
\begin{equation}\label{eq:cyclicharr}
	HC_{\mathrm{Harr}}^{i+1,*}(H^*)
	\cong
	\operatorname{Harr}^{i,*}(H^*,H^*),
	\qquad i\geq3,
\end{equation}
where the second grading on the two sides is related by the pairing
shift just described.

Let
\[
\mathfrak g_{H^*,\mathrm{cyc},\mathrm{unit}}
\subseteq
\mathfrak g_{H^*}
\]
denote the complete filtered sub-DGLA consisting of normalized cyclic
Harrison cochains.

\begin{theorem}\label{thm:iso_from_equiv}
	Let $A_1$ and $A_2$ be two minimal $C_\infty$-algebra enhancements
	of a graded commutative algebra $H^*$ over a field $\mathbb F$ of
	characteristic zero. Assume that
	\[
	\mathcal F^1\mathfrak g_{H^*}^0/
	\mathcal F^k\mathfrak g_{H^*}^0
	\]
	is finite-dimensional for every $k$; in particular, this holds when
	$H^*$ is finite-dimensional.
	
	\begin{enumerate}
		\item[\rm(a)] The following conditions are equivalent:
		\begin{enumerate}
			\item $A_1$ and $A_2$ are $C_\infty$-isotopic as minimal $C_\infty$-algebras;
			\item $A_1$ and $A_2$ are isotopic modulo $k$ for every $k\geq3$.
		\end{enumerate}
		
		\item[\rm(b)] Suppose, moreover, that $H^*$ is a connected PGCA and
		that $A_1,A_2$ are minimal unital cyclic $C_\infty$-algebra
		enhancements. Then the following conditions are equivalent:
		\begin{enumerate}
			\item $A_1$ and $A_2$ are $C_\infty$-isotopic as minimal $C_\infty$-algebras;
			\item $A_1$ and $A_2$ are isotopic as minimal unital cyclic
			$C_\infty$-algebras; equivalently, there exists
			\[
			p\in
			\mathcal F^1
			\mathfrak g_{H^*,\mathrm{cyc},\mathrm{unit}}^0
			\]
			such that
			\[
			m^{(2)}
			=
			e^{\operatorname{ad}_p}*m^{(1)};
			\]
			\item $A_1$ and $A_2$ are isotopic modulo $k$ as minimal unital
			cyclic $C_\infty$-algebras for every $k\geq3$.
		\end{enumerate}
	\end{enumerate}
\end{theorem}

\begin{proof}
	(a) The implication $(1)\Rightarrow(2)$ is immediate.
	
	Let
	\[
	\mathcal G_k
	\coloneqq
	\exp\left(
	\mathcal F^1\mathfrak g_{H^*}^0/
	\mathcal F^k\mathfrak g_{H^*}^0
	\right)
	\]
	be the truncated gauge group, and denote by $\bar m_k$ and $\bar m_k'$ the
	images of $m$ and $m'$ in
	$\mathfrak g_{H^*}/\mathcal F^k\mathfrak g_{H^*}$. Set
	\[
	G_k(m,m')
	\coloneqq
	\left\{
	g\in\mathcal G_k:
	g*\bar m_k=\bar m_k'
	\right\}.
	\]
	Then
	\[
	m\sim_km'
	\quad\Longleftrightarrow\quad
	G_k(m,m')\neq\varnothing .
	\]
	The quotient maps of the filtered DGLA induce compatible truncation
	maps
	\[
	\rho_{N,k}:G_N(m,m')\longrightarrow G_k(m,m'),
	\qquad N\geq k.
	\]
	
	The complete gauge group satisfies
	\[
	\exp(\mathcal F^1\mathfrak g_{H^*}^0)
	\cong
	\varprojlim_k\mathcal G_k.
	\]
	Consequently, a full isotopy from $m$ to $m'$ is equivalent to an
	element of
	\[
	\varprojlim_kG_k(m,m').
	\]
	
	For each $k$, let
	\[
	S_k(m')
	\coloneqq
	\left\{
	h\in\mathcal G_k:
	h*\bar m_k'=\bar  m_k'
	\right\}
	\]
	be the truncated stabilizer. Whenever $G_k(m,m')$ is nonempty, it is
	a left torsor under $S_k(m')$: if $g,g'\in G_k(m,m')$, then
	\[
	g'g^{-1}\in S_k(m'),
	\qquad
	g'=(g'g^{-1})g.
	\]
	
	Fix $k$ and, for $N\geq k$, define
	\[
	G_k^{(N)}(m,m')
	\coloneqq
	\operatorname{Im}\bigl(G_N(m,m')\to G_k(m,m')\bigr)
	\]
	and
	\[
	S_k^{(N)}(m')
	\coloneqq
	\operatorname{Im}\bigl(S_N(m')\to S_k(m')\bigr).
	\]
	The truncation homomorphisms
	\[
	S_N(m')\longrightarrow S_k(m'),
	\qquad N\geq k,
	\]
	are morphisms of algebraic groups. Hence their images
	\[
	S_k^{(N)}(m')
	\coloneqq
	\operatorname{Im}\bigl(S_N(m')\to S_k(m')\bigr)
	\]
	are closed algebraic subgroups of $S_k(m')$ by
	\cite[Proposition B(b), Subsection 7.4]{Humphreys1975}.
	They form a descending sequence
	\[
	S_k^{(k)}(m')
	\supseteq
	S_k^{(k+1)}(m')
	\supseteq\cdots .
	\]
	Since $S_k(m')$ is a finite-dimensional unipotent algebraic group
	over a field of characteristic zero, all its algebraic subgroups are
	connected. Therefore, if
	\[
	S_k^{(N+1)}(m')
	\subseteq
	S_k^{(N)}(m')
	\]
	is a proper inclusion, then
	\[
	\dim S_k^{(N+1)}(m')
	<
	\dim S_k^{(N)}(m').
	\]
	The nonnegative integers
	\[
	\dim S_k^{(N)}(m')
	\]
	can decrease only finitely many times. Consequently, the descending
	sequence $S_k^{(N)}(m')$ eventually stabilizes.

	Since $G_N(m,m')$ is a torsor under $S_N(m')$, its image
	$G_k^{(N)}(m,m')$ is a left coset of $S_k^{(N)}(m')$. The sets
	$G_k^{(N)}(m,m')$ form a nested sequence of nonempty cosets. Once the
	corresponding subgroups $S_k^{(N)}(m')$ have stabilized, these nested
	cosets are equal. Therefore, for every fixed $k$, the images
	\[
	\operatorname{Im}\bigl(G_N(m,m')\to G_k(m,m')\bigr)
	\]
	stabilize as $N\to\infty$.
	
	Thus the inverse system $\{G_k(m,m')\}$ satisfies the
	Mittag--Leffler condition. Since each $G_k(m,m')$ is nonempty by
	assumption, the standard inverse-limit lemma for nonempty sets yields
	\[
	\varprojlim_kG_k(m,m')\neq\varnothing.
	\]
	An element of this inverse limit determines a compatible family of
	truncated gauge transformations and hence, by completeness, a full
	$C_\infty$-isotopy from $m$ to $m'$. This proves
	$(2)\Rightarrow(1)$.
	
	(b) Hamilton--Lazarev prove that, for two minimal symplectic
	$C_\infty$-structures with the same underlying Frobenius algebra,
	every homotopy class of pointed $C_\infty$-morphisms has a symplectic
	representative, and that in the connected unital case every
	symplectic homotopy class has a unital symplectic representative;
	see \cite[Theorems 5.5(ii) and 5.10(ii)]{HL2008}. Hence
	$(1)\Leftrightarrow(2)$.
	
	The implication $(2)\Rightarrow(3)$ is immediate. Conversely,
	$\mathfrak g_{H^*,\mathrm{cyc},\mathrm{unit}}$ is a complete filtered
	sub-DGLA, and its finite truncations are finite-dimensional because
	$H^*$ is a PGCA. Applying the Mittag--Leffler argument from part
	{\rm(a)} to this sub-DGLA gives $(3)\Rightarrow(2)$.
\end{proof}

\begin{remark}\label{rem:infinite-dimensional-H}
	The finite-dimensionality of $H^*$ is not needed for most of the
	constructions and results in this subsection. For an arbitrary graded
	commutative algebra $H^*$ over a field of characteristic zero, the
	completed Harrison DGLA
	\[
	\mathfrak g_{H^*}
	=	
	\prod_{q\geq2}
	C_{\mathrm{Harr}}^{q,*}(H^*,H^*)
	\]
	is complete with respect to the shifted arity filtration, and its
	Maurer--Cartan elements describe minimal $C_\infty$-algebra
	enhancements of $H^*$. Likewise, the gauge action describes
	$C_\infty$-isotopies, and the notions of isotopy modulo $k$, the
	corresponding obstruction sets, and their additivity remain valid
	without any finite-dimensionality assumption.
	
	Finite-dimensionality is used in the proof of
	Theorem~\ref{thm:iso_from_equiv}, where one passes from isotopies
	modulo $k$ for every $k$ to a single $C_\infty$-isotopy. More
	precisely, it ensures the required stabilization, or
	Mittag--Leffler property, for the inverse system of finite-stage gauge
	data. Thus Theorem~\ref{thm:iso_from_equiv} remains valid for an
	arbitrary $H^*$ whenever the corresponding inverse system satisfies
	this Mittag--Leffler condition.
\end{remark}

	\subsection{The primary obstruction: isotopy modulo $3$}\label{subs:probs}
	The first step in comparing two minimal $C_\infty$-algebra
	enhancements is to analyze isotopy modulo $3$. Since
	\[
	m_3\in C_{\mathrm{Harr}}^{3,-1}(H^*,H^*)
	\]
	and the arity-$4$ component of the Maurer--Cartan equation gives
	$\delta m_3=0$, the operation $m_3$ determines a Harrison
	cohomology class. The following proposition shows that this class
	completely characterizes isotopy modulo $3$.
	
	\begin{proposition}[Primary obstruction]\label{prop:equiv3}
		Let $H^*$ be a  GCA, and let
		\[
		A_1=(H^*,\{m_i^{(1)}\}),
		\qquad
		A_2=(H^*,\{m_i^{(2)}\})
		\]
		be two minimal $C_\infty$-algebra enhancements of $H^*$. Then
		\begin{equation}\label{eq:equiv3}
			A_1\sim_3 A_2
			\quad\Longleftrightarrow\quad
			[m_3^{(1)}]=[m_3^{(2)}]
			\quad\text{in}\quad
			\operatorname{Harr}^{3,-1}(H^*,H^*).
		\end{equation}
		
		Suppose, moreover, that $H^*$ is a PGCA and that $A_1$ and $A_2$
		are minimal unital cyclic $C_\infty$-algebra enhancements. Under the
		identification between the normalized cyclic Harrison complex and
		the normalized Harrison complex induced by the Poincar\'e pairing,
		\begin{equation}\label{eq:equiv3a}
			A_1\sim_3 A_2
			\quad\Longleftrightarrow\quad
			[Tm_3^{(1)}]=[Tm_3^{(2)}]
			\quad\text{in}\quad
			HC_{\mathrm{Harr}}^{4,-1}(H^*).
		\end{equation}
		Here the second index on the cyclic complex is the degree inherited
		from the corresponding $H^*$-valued Harrison cochain.
	\end{proposition}

	\begin{proof}
		The arity-$4$ component of the Maurer--Cartan equation
		\[
		\delta m+\frac12[m,m]=0
		\]
		is
		\[
		\delta m_3=0,
		\]
		because the bracket term has arity at least $5$. Thus
		$m_3^{(1)}$ and $m_3^{(2)}$ define classes in
		$\operatorname{Harr}^{3,-1}(H^*,H^*)$.
		
		Assume first that $A_1\sim_3 A_2$. Then there exists a
		$C_\infty$-isotopy
		\[
		\phi=(\phi_1=\Id,\phi_2,\phi_3,\ldots)
		\]
		such that, for the transported structure
		\[
		A_1''=\phi(A_1),
		\]
		we have
		\[
		m_3^{(1'')}=m_3^{(2)}.
		\]
		The arity-$3$ transformation formula gives
		\begin{equation}\label{eq:phi2m3}
			m_3^{(1'')}
			=
			m_3^{(1)}-\delta\phi_2.
		\end{equation}
		Consequently,
		\[
		m_3^{(1)}-m_3^{(2)}
		=
		\delta\phi_2,
		\]
		and hence
		\[
		[m_3^{(1)}]=[m_3^{(2)}].
		\]
		
		Conversely, assume that
		\[
		[m_3^{(1)}]=[m_3^{(2)}].
		\]
		Then there exists
		\[
		\phi_2\in C_{\mathrm{Harr}}^{2,-1}(H^*,H^*)
		\]
		such that
		\[
		m_3^{(1)}-m_3^{(2)}
		=
		\delta\phi_2.
		\]
		The sequence
		\[
		\phi=(\Id,\phi_2,0,0,\ldots)
		\]
		defines an invertible pointed formal map. Transport the
		$C_\infty$-structure $A_1$ along $\phi$ and denote the resulting
		structure by $A_1''=\phi(A_1)$. Then $\phi$ is a
		$C_\infty$-isotopy from $A_1$ to $A_1''$. Since the binary
		multiplication is unchanged and \eqref{eq:phi2m3} gives
		\[
		m_3^{(1'')}
		=
		m_3^{(1)}-\delta\phi_2
		=
		m_3^{(2)},
		\]
		we obtain
		\[
		\phi(A_1)=A_2\pmod 3.
		\]
		Thus $A_1\sim_3 A_2$, proving \eqref{eq:equiv3}.
		
		In the unital cyclic case, the operations $m_3^{(1)}$ and
		$m_3^{(2)}$ are normalized cyclic Harrison cochains. The
		Poincar\'e pairing identifies their cyclic cohomology classes with
		their corresponding normalized Harrison cohomology classes.
		Therefore \eqref{eq:equiv3a} follows from \eqref{eq:equiv3} and the
		cyclic--Harrison comparison isomorphism.
	\end{proof}

\subsection{Higher obstructions: isotopy modulo $k\geq4$}\label{subs:secondobs}

Let us begin with  the secondary obstruction  for extending isotopy  modulo 3 to isotopy  modulo 4.  
Assume that $m\sim_3m'$. Choose a gauge parameter
\[
p=p_2+p_3+\cdots\in\mathcal F^1\mathfrak g^0
\]
such that
\begin{equation}
e^{\operatorname{ad}_p}*m
\equiv
m'
\pmod{\mathcal F^4}.\label{eq:equiv4}
\end{equation}
Here
\[
p_i\in C_{\mathrm{Harr}}^{i,1-i}(H^*,H^*).
\]
The arity-$3$ component is
\[
m_3'=m_3-\delta p_2.
\]
The arity-$4$ component of the gauge formula is
\begin{equation}\label{eq:tk4}
m_4'
=
m_4+[p_2,m_3]-\delta p_3
-\frac12[p_2,\delta p_2].
\end{equation}
Thus define
\begin{equation}
\widetilde\kappa_4(p_2;m,m')
\coloneqq
m_4-m_4'+[p_2,m_3]
-\frac12[p_2,\delta p_2].\label{eq:tildek4}
\end{equation}

\begin{lemma}\label{lem:2obst}
	The cochain
	\[
	\widetilde\kappa_4(p_2;m,m')
	\in
	C_{\mathrm{Harr}}^{4,-2}(H^*,H^*)
	\]
	is a Harrison cocycle.
\end{lemma}

\begin{proof}
	The arity-$5$ components of the Maurer--Cartan equations for $m$ and
	$m'$ give
	\[
	\delta m_4=-\frac12[m_3,m_3],
	\qquad
	\delta m_4'=-\frac12[m_3',m_3'].
	\]
	Writing $d=\delta p_2=m_3-m_3'$ and using that $m_3$ and $d$ have
	DGLA degree one, we obtain
	\[
	\delta(m_4-m_4')
	=
	-[m_3,d]+\frac12[d,d].
	\]
	Moreover,
	\[
	\delta[p_2,m_3]=[d,m_3]=[m_3,d],
	\]
	and
	\[
	\delta[p_2,d]=[d,d].
	\]
	Hence
	\[
	\delta\widetilde\kappa_4(p_2;m,m')=0.
	\]
\end{proof}

For a fixed
\[
p_2\in C_{\mathrm{Harr}}^{2,-1}(H^*,H^*)
\qquad\text{with}\qquad
\delta p_2=m_3-m_3',  
\]
 by \eqref{eq:tk4}  there exists
\[
p_3\in C_{\mathrm{Harr}}^{3,-2}(H^*,H^*)
\]
such that
\[
e^{\operatorname{ad}_{p_2+p_3+\cdots}}*m
\equiv
m'
\pmod{\mathcal F^4}
\]
if and only if
\[
\bigl[
\widetilde\kappa_4(p_2;m,m')
\bigr]
=
0
\quad\text{in}\quad
\operatorname{Harr}^{4,-2}(H^*,H^*),
\]
where  $\widetilde\kappa_4(p_2;m,m')$ is defined by \eqref{eq:tildek4}.

\medskip

We now formulate the higher obstruction theory.  The proof applies
more generally to a complete filtered DGLA
\[
(\mathfrak g,\delta,[-,-],\mathcal F)
\]
satisfying
\[
\delta(\mathcal F^r)\subseteq\mathcal F^{r+1},
\qquad
[\mathcal F^r,\mathcal F^s]\subseteq\mathcal F^{r+s}.
\]
Write
\[
\operatorname{gr}_{\mathcal F}^r\mathfrak g
\coloneqq
\mathcal F^r\mathfrak g/\mathcal F^{r+1}\mathfrak g.
\]
The differential induces maps
\[
\bar\delta:
\operatorname{gr}_{\mathcal F}^r\mathfrak g
\longrightarrow
\operatorname{gr}_{\mathcal F}^{r+1}\mathfrak g.
\]
For the shifted arity filtration on the Harrison DGLA,
\[
\operatorname{gr}_{\mathcal F}^{k-1}\mathfrak g^1
\cong
C_{\mathrm{Harr}}^{k,2-k}(H^*,H^*),
\]
and $\bar\delta$ is the Harrison differential.

Let
\[
m,m'\in\mathcal F^2\mathfrak g^1
\]
be Maurer--Cartan elements and suppose that
\[
m\sim_{k-1}m',
\qquad k\geq4.
\]
Define
\[
\mathcal P_{k-1}(m,m')
\coloneqq
\left\{
p\in\mathcal F^1\mathfrak g^0:
e^{\operatorname{ad}_p}*m-m'
\in\mathcal F^{k-1}\mathfrak g
\right\}.
\]
Let
\[
\pi_k:
\mathcal F^{k-1}\mathfrak g
\longrightarrow
\operatorname{gr}_{\mathcal F}^{k-1}\mathfrak g
\]
be the quotient map. For
$p\in\mathcal P_{k-1}(m,m')$, set
\begin{equation}\label{eq:mmk}
	\widetilde\kappa_k(p;m,m')
	\coloneqq
	\pi_k\left(e^{\operatorname{ad}_p}*m-m'\right).
\end{equation}

\begin{lemma}[Closedness]\label{lem:closed}
	For every $p\in\mathcal P_{k-1}(m,m')$, one has
	\[
	\bar\delta\,
	\widetilde\kappa_k(p;m,m')
	=
	0.
	\]
	In the Harrison DGLA, $\widetilde\kappa_k(p;m,m')$ is therefore a
Harrison cocycle in
\[
C_{\mathrm{Harr}}^{k,2-k}(H^*,H^*).
\]
\end{lemma}

\begin{proof}
	Set
	\[
	M=e^{\operatorname{ad}_p}*m,
	\qquad
	D=M-m'\in\mathcal F^{k-1}\mathfrak g^1.
	\]
	Since $M$ and $m'$ are Maurer--Cartan elements, subtracting their
	Maurer--Cartan equations gives
	\[
	\delta D+[m',D]+\frac12[D,D]=0.
	\]
	Now
	\[
	[m',D]\in
	[\mathcal F^2,\mathcal F^{k-1}]
	\subseteq
	\mathcal F^{k+1},
	\]
	and
	\[
	[D,D]\in\mathcal F^{2k-2}
	\subseteq\mathcal F^{k+1}.
	\]
	Hence
	\[
	\delta D\in\mathcal F^{k+1},
	\]
	which is precisely
	\[
	\bar\delta\,\pi_k(D)=0.
	\]
\end{proof}

\begin{definition}\label{def:k-obstr}
	The $k$-th obstruction set is
	\[
	\mathcal K_k(m,m')
	\coloneqq
	\left\{
	[\widetilde\kappa_k(p;m,m')]:
	p\in\mathcal P_{k-1}(m,m')
	\right\}
	\subseteq
	\operatorname{Harr}^{k,2-k}(H^*,H^*).
	\]
\end{definition}

\begin{proposition}[Extension criterion]\label{prop:extcri}
	One has
	\[
	m\sim_km'
	\quad\Longleftrightarrow\quad
	0\in\mathcal K_k(m,m').
	\]
	More precisely, let
	$p\in\mathcal P_{k-1}(m,m')$.  The isotopy represented by $p$
	can be extended through arity $k$, without changing its components
	below arity $k-1$, if and only if
	\[
	[\widetilde\kappa_k(p;m,m')]=0.
	\]
\end{proposition}

\begin{proof}
	If $m\sim_km'$, choose $p$ such that
	\[
	e^{\operatorname{ad}_p}*m-m'\in\mathcal F^k.
	\]
	Then
	\[
	\widetilde\kappa_k(p;m,m')=0.
	\]
	
	Conversely, suppose that
	\[
	[\widetilde\kappa_k(p;m,m')]=0.
	\]
	Then there is a class
	\[
	\bar u\in
	\operatorname{gr}_{\mathcal F}^{k-2}\mathfrak g^0
	\]
	such that
	\[
	\bar\delta\bar u
	=
	\widetilde\kappa_k(p;m,m').
	\]
	Choose a representative
	\[
	u\in\mathcal F^{k-2}\mathfrak g^0
	\]
	of $\bar u$, and set
	\[
	\bar p\coloneqq\operatorname{BCH}(u,p),
	\]
	so that
	\[
	\Gamma_{\bar p}
	=
	\Gamma_u\circ\Gamma_p,
	\qquad
	\Gamma_v(x)\coloneqq e^{\operatorname{ad}_v}*x.
	\]
	Writing
	\[
	e^{\operatorname{ad}_p}*m=m'+D,
	\qquad
	D\in\mathcal F^{k-1},
	\]
	the affine gauge formula gives, modulo $\mathcal F^k$,
	\[
	e^{\operatorname{ad}_u}*(m'+D)-m'
	\equiv
	D-\delta u.
	\]
	The image of the right-hand side in
	$\operatorname{gr}_{\mathcal F}^{k-1}\mathfrak g$ is zero by the
	choice of $u$. Therefore
	\[
	e^{\operatorname{ad}_{\bar p}}*m-m'
	\in\mathcal F^k,
	\]
	and hence $m\sim_km'$.
\end{proof}

\begin{proposition}[Additivity of obstruction classes]
	\label{prop:obstruction-additivity}
	Let
	\[
	m^{(0)},m^{(1)},m^{(2)}\in\mathcal F^2\mathfrak g^1
	\]
	be Maurer--Cartan elements. Suppose that
	\[
	p\in\mathcal P_{k-1}(m^{(0)},m^{(1)}),
	\qquad
	q\in\mathcal P_{k-1}(m^{(1)},m^{(2)}).
	\]
	Define
	\[
	q*p\coloneqq\operatorname{BCH}(q,p),
	\]
	so that
	\[
	\Gamma_{q*p}=\Gamma_q\circ\Gamma_p.
	\]
	Then
	\[
	q*p\in\mathcal P_{k-1}(	m^{(0)}, m^{(2)})
	\]
	and
	\begin{equation}\label{eq:obstruction-additivity}
		\widetilde\kappa_k(q*p;m^{(0)}, m^{(2)})
		=
		\widetilde\kappa_k(p;m^{(0)},m^{(1)})
		+
		\widetilde\kappa_k(q;m^{(1)},m^{(2)})
	\end{equation}
	in
	$\operatorname{gr}_{\mathcal F}^{k-1}\mathfrak g^1$.
\end{proposition}

\begin{proof}
	Set
	\[
	D_p=\Gamma_p(m^{(0)})-m^{(1)},
	\qquad
	D_q=\Gamma_q(m^{(1)})-m^{(2)}.
	\]
	Then
	\[
	D_p,D_q\in\mathcal F^{k-1}.
	\]
	Using the affine gauge formula,
	\begin{align*}
		\Gamma_{q*p}(m^{(0)})-m^{(2)}
		&=
		\Gamma_q\bigl(m^{(1)}+D_p\bigr)-m^{(2)}\\
		&=
		D_q+e^{\operatorname{ad}_q}(D_p).
	\end{align*}
	Since
	$q\in\mathcal F^1\mathfrak g^0$ and
	$D_p\in\mathcal F^{k-1}\mathfrak g^1$, we have
	\[
	e^{\operatorname{ad}_q}(D_p)
	\equiv
	D_p
	\pmod{\mathcal F^k}.
	\]
	Projecting to
	$\operatorname{gr}_{\mathcal F}^{k-1}\mathfrak g^1$
	gives \eqref{eq:obstruction-additivity}.
\end{proof}

\begin{corollary}[Stabilizers and the torsor of obstructions]
	\label{cor:obstruction-torsor}
	Assume that $m\sim_{k-1}m'$.
	
	\begin{enumerate}
		\item The set
$
		\mathcal K_k(m,m)
$
		is an additive subgroup of
		$\operatorname{Harr}^{k,2-k}(H^*,H^*)$.
		
		\item The stabilizer obstruction subgroups agree:
		\[
		\mathcal K_k(m,m)
		=
		\mathcal K_k(m',m').
		\]
		
		\item The obstruction set
		\[
		\mathcal K_k(m,m')
		\]
		is a torsor, equivalently a coset, under this common additive
		subgroup.
	\end{enumerate}
\end{corollary}

\begin{proof}
	The identity gauge transformation represents zero. Closure under
	addition follows from
	Proposition~\ref{prop:obstruction-additivity}. If
	$p\in\mathcal P_{k-1}(m,m)$, then the inverse gauge transformation
	is represented by $-p$, and additivity applied to
	\[
	(-p)*p=0
	\]
	shows that its obstruction class is the negative of the class of
	$p$. This proves (1).
	
	Choose
	\[
	p\in\mathcal P_{k-1}(m,m').
	\]
	If
	$q\in\mathcal P_{k-1}(m,m)$, let
	\[
	r\coloneqq p*q*(-p).
	\]
	Then
	$r\in\mathcal P_{k-1}(m',m')$. Repeated application of
	Proposition~\ref{prop:obstruction-additivity} gives
	\[
	[\widetilde\kappa_k(r;m',m')]
	=
	[\widetilde\kappa_k(-p;m',m)]
	+
	[\widetilde\kappa_k(q;m,m)]
	+
	[\widetilde\kappa_k(p;m,m')].
	\]
	Since
	\[
	p*(-p)=0
	\]
	as a gauge transformation from $m'$ to itself, the first and third
	terms cancel. Hence
	\[
	[\widetilde\kappa_k(r;m',m')]
	=
	[\widetilde\kappa_k(q;m,m)].
	\]
	This proves
	\[
	\mathcal K_k(m,m)
	\subseteq
	\mathcal K_k(m',m').
	\]
	The reverse inclusion follows by interchanging $m$ and $m'$, proving
	(2).
	
	Finally, fix
	\[
	p_0\in\mathcal P_{k-1}(m,m').
	\]
	Proposition~\ref{prop:obstruction-additivity} shows that composition
	on the left by a stabilizer of $m'$ adds its obstruction class.
	Conversely, if
	$p\in\mathcal P_{k-1}(m,m')$, then
	\[
	q\coloneqq p*(-p_0)
	\]
	belongs to
	$\mathcal P_{k-1}(m',m')$, and
	\[
	[\widetilde\kappa_k(p;m,m')]
	=
	[\widetilde\kappa_k(p_0;m,m')]
	+
	[\widetilde\kappa_k(q;m',m')].
	\]
	Thus
	\[
	\mathcal K_k(m,m')
	=
	[\widetilde\kappa_k(p_0;m,m')]
	+
	\mathcal K_k(m',m'),
	\]
	which proves (3).
\end{proof}

\begin{theorem}[Higher obstruction classes]\label{thm:additivek}
	Let $m$ and $m'$ be two minimal $C_\infty$-algebra enhancements of
	a GCA $H^*$ and assume that
	\[
	m\sim_{k-1}m',
	\qquad k\geq4.
	\]
	Then:
	
	\begin{enumerate}
		\item The elements defining $\mathcal K_k(m,m')$ are Harrison
cocycles of bidegree $(k,2-k)$. More precisely, they lie in
\[
C_{\mathrm{Harr}}^{k,2-k}(H^*,H^*).
\]
The obstruction set depends only on the truncations of $m$ and $m'$
modulo $\mathcal F^k$; equivalently, it depends only on their operations
through arity $k$.
		
		\item One has
		\[
		0\in\mathcal K_k(m,m')
		\quad\Longleftrightarrow\quad
		m\sim_km'.
		\]
		
		\item The sets
		\[
		\mathcal K_k(m,m)
		\quad\text{and}\quad
		\mathcal K_k(m',m')
		\]
		are equal additive subgroups of
		$\operatorname{Harr}^{k,2-k}(H^*,H^*)$.
		
		\item The obstruction set
		$\mathcal K_k(m,m')$ is a torsor under this common subgroup.
	\end{enumerate}
\end{theorem}

\begin{proof}
	The assertions follow from
	Lemma~\ref{lem:closed},
	Proposition~\ref{prop:extcri},
	Proposition~\ref{prop:obstruction-additivity}, and
	Corollary~\ref{cor:obstruction-torsor}.
\end{proof}

\medskip

Let $r\geq2$. We are now ready to classify the real and rational homotopy types of $(r-1)$-connected closed smooth manifolds of dimension $n\leq\ell(r-1)+2$, where $\ell\geq4$. 

Denote by
$C_{\mathrm{Harr,cyc,unit}}^{*,*}(H^*,H^*)$ the subcomplex of the
Harrison cochain complex $C^{*,*}(H^*,H^*)$ consisting of normalized
cyclic Harrison cochains, and denote by
$\mathfrak g_{H^*,\mathrm{cyc,unit}}$ the corresponding sub-DGLA of
$\mathfrak g_{H^*}$.
Hamilton-Lazarev  proved that     for  a Poincar\'e   GCA  $H$   there is 1-1  correspondence   between   the set  $C_\infty (H^*)$ of  all     isotopy  classes of minimal  $C_\infty$-algebra enhancements  of $H^*$ and  the set $C_{\infty, \,cyc, \, unit }(H^*)$  of all  unital cyclic  isotopy  classes of minimal unital   cyclic  $C_\infty$-algebra  enhancements  of $H^*$ \cite[Theorems 5.5, 5.10]{HL2008}.

\begin{theorem}\label{thm:finite}
Let $\F=\Q$ or $\F=\R$, and let $M$ be a closed
$(r-1)$-connected manifold of dimension
\[
n\leq \ell(r-1)+2,
\qquad r\geq2,
\qquad \ell\geq4.
\]
Then the $\F$-homotopy type of $M$ is determined uniquely by its
cohomology algebra
\[
H^*\coloneqq H^*(M;\F)
\]
and by the isotopy class modulo $(\ell-2)$, in the category of minimal
unital cyclic $C_\infty$-algebras, of the corresponding enhancement of
$H^*$.

In particular:
\begin{enumerate}
\item if $n\leq5r-3$, the $\F$-homotopy type is determined by $H^*$
and the primary cyclic Harrison class
\[
[Tm_3]\in HC_{\mathrm{Harr}}^{4,-1}(H^*);
\]
\item if $n\leq6r-4$, and $M'$ is another closed
$(r-1)$-connected manifold together with a fixed graded-algebra
isomorphism
\[
\theta\colon H^*(M;\F)\xrightarrow{\cong}H^*(M';\F),
\]
then $M$ and $M'$ have the same $\F$-homotopy type if and only if,
after using $\theta$ to regard both enhancements as structures on the
same Poincar\'e algebra, their primary classes
agree and
\[
0\in\mathcal K_4(m,m').
\]
\end{enumerate}
\end{theorem}

We begin with the following degree restriction.

\begin{lemma}[Degree restriction for normalized cyclic Harrison cochains]
\label{lem:degree}
Let $H^*$ be an $(r-1)$-connected PGCA of degree $n$, where
$r\geq2$ and
\[
n\leq\ell(r-1)+2,
\qquad \ell\geq4.
\]
Then
\[
C_{\mathrm{Harr,cyc,unit}}^{k,2-k}(H^*,H^*)=0
\qquad\text{for every }k\geq\ell-1.
\]
\end{lemma}

\begin{proof}
Let
\[
\phi\in C_{\mathrm{Harr,cyc,unit}}^{k,2-k}(H^*,H^*)
\]
and suppose that
$\phi(a_1,\ldots,a_k)\neq0$ for homogeneous inputs
$a_i\in H^{d_i}$. Since $\phi$ is normalized, none of the inputs can
be the unit; hence $d_i\geq r$ for $1\leq i\leq k$.

By nondegeneracy of the Poincar\'e pairing, there is a homogeneous
$a_{k+1}\in H^{d_{k+1}}$ such that
\[
\bigl\langle\phi(a_1,\ldots,a_k),a_{k+1}\bigr\rangle\neq0.
\]
Cyclicity and normalization imply
\[
\bigl\langle\phi(a_1,\ldots,a_k),1\bigr\rangle=0,
\]
so $a_{k+1}$ is not the unit and therefore $d_{k+1}\geq r$. Degree
compatibility with the Poincar\'e pairing gives
\[
\sum_{i=1}^{k+1}d_i=n+k-2.
\]
Consequently,
\[
(k+1)r
\leq n+k-2
\leq \ell(r-1)+k,
\]
and hence
\[
(k+1)(r-1)+1\leq\ell(r-1).
\]
Since $r\geq2$, this implies $k+1<\ell$, and therefore
$k\leq\ell-2$. Thus no nonzero cochain of the indicated bidegree can
exist when $k\geq\ell-1$.
\end{proof}

\begin{proof}[Proof of Theorem \ref{thm:finite}]
Kadeishvili's classification identifies the $\F$-homotopy type with
the $C_\infty$-isotopy class of the minimal enhancement of $H^*$.
By the Hamilton--Lazarev comparison theorem, this class may be
represented and compared in the category of minimal unital cyclic
$C_\infty$-algebras. Lemma \ref{lem:degree} shows that every operation
of arity at least $\ell-1$ vanishes in this category. Therefore an
isotopy modulo $(\ell-2)$ already identifies the full cyclic
$C_\infty$-structures, and Theorem \ref{thm:iso_from_equiv} yields the
first assertion.

For $n\leq5r-3$, take $\ell=5$ and use Proposition
\ref{prop:equiv3}. For $n\leq6r-4$, take $\ell=6$ and combine the
primary comparison with the extension criterion
(Proposition \ref{prop:extcri}), which says precisely that the isotopy
extends through arity $4$ if and only if
$0\in\mathcal K_4(m,m')$.
\end{proof}

 \begin{corollary}\label{cor:Zhoucyclic} 
 		Let $r\geq 2$ and let $\Aa^\ast$ be a $(r-1)$-connected Poincar\'e DGCA of degree $n$ admitting a Hodge homotopy.
 		If $n\leq\ell(r-1)+2$, with $\ell\geq4$, then $H^* (\Aa^*)$ carries a minimal unital  cyclic $C_\infty$-algebra enhancement making it $C_\infty$-quasi-isomorphic to $\Aa^\ast$,  and whose multiplications $m_k$ vanish for $k\geq\ell-1$. In particular, all the multiplications $m_k$ with $k\geq 5$ vanish if $n\leq 6r-4$ and  all the multiplications $m_k$ with $k\geq 4$ vanish if $n\leq 5r-3$. Also, all the multiplications $m_k$ with $k\geq 3$ vanish if $n\leq 4r-2$, and so  $\Aa^\ast$ is formal in this case.
 \end{corollary}
 
Corollary \ref{cor:Zhoucyclic} strengthens the formulation of \cite[Theorem 2.12]{FiorenzaLe2025} by showing that the transferred structure may be taken to be cyclic.

\subsection{Example: the Fern\'andez--Mu\~noz $8$-manifold}
\label{subs:example}

We illustrate the secondary obstruction by the simply connected
compact symplectic $8$-manifold $\widetilde M$ constructed by
Fern\'andez and Mu\~noz in \cite{FM2008}.  The manifold
$\widetilde M$ is obtained by resolving the singularities of the
orbifold
\[
M/\mathbb Z_3,
\]
where
\[
M=\Gamma\backslash(H_{\mathbb C}\times\mathbb C)
\]
is a nilmanifold that is a principal $T^2$-bundle over $T^6$.

The explicit calculation in \cite[Theorem 3.2]{FM2008} is most
naturally written with complex coefficients.  On the nilmanifold
$M$, let
\[
\alpha=\mu\wedge\bar\mu,
\qquad
\beta_1=\nu\wedge\bar\nu,
\qquad
\beta_2=\nu\wedge\bar\eta,
\qquad
\beta_3=\bar\nu\wedge\eta .
\]
These forms are closed and $\mathbb Z_3$-invariant.  Moreover,
\[
\alpha\wedge\beta_i=d\xi_i,
\qquad i=1,2,3,
\]
where one may take
\[
\xi_1=-\theta\wedge\bar\mu\wedge\bar\nu,
\qquad
\xi_2=-\theta\wedge\bar\mu\wedge\bar\eta,
\qquad
\xi_3=\bar\theta\wedge\mu\wedge\eta .
\]
After descending these forms to $M/\mathbb Z_3$, modifying them near
the singular points, and extending them across the exceptional
divisors, Fern\'andez and Mu\~noz obtain classes on $\widetilde M$
for which
\begin{equation}\label{eq:FM-higher-product}
	\left[
	\xi_1\wedge\xi_2\wedge\beta_3
	+
	\xi_2\wedge\xi_3\wedge\beta_1
	+
	\xi_3\wedge\xi_1\wedge\beta_2
	\right]
	\neq0
	\quad\text{in}\quad
	H^8(\widetilde M;\mathbb C).
\end{equation}
More precisely, before descent and resolution, the form representing
\eqref{eq:FM-higher-product} is
\[
2\theta\wedge\mu\wedge\nu\wedge\eta
\wedge\bar\theta\wedge\bar\mu\wedge\bar\nu\wedge\bar\eta,
\]
whose integral is nonzero.  Their higher-product criterion therefore
shows that $\widetilde M$ is non-formal.

We now interpret this example in terms of the obstruction theory
developed above.  Fern\'andez and Mu\~noz prove that \cite[Lemma 2.4]{FM2008}
\[
H^{2j+1}(\widetilde M;\mathbb Q)=0
\qquad
\text{for every }j.
\]
Let
\[
m=m_3+m_4+\cdots
\]
be a minimal unital cyclic $C_\infty$-algebra enhancement of
$H^*(\widetilde M;\mathbb Q)$.  Since $m_3$ has internal degree
$-1$ and the cohomology of $\widetilde M$ is concentrated in even
degrees, we necessarily have
\[
m_3=0.
\]
Thus the primary obstruction vanishes.

On the other hand, $\widetilde M$ is simply connected and
$8$-dimensional.  Applying Lemma~\ref{lem:degree} with
$r=2$ and $\ell=6$ shows that
\[
m_k=0
\qquad
\text{for every }k\geq5.
\]
Consequently, the only possible nontrivial higher multiplication is
$m_4$.  Since $m_3=0$, the Maurer--Cartan equation implies
\[
\delta m_4=0.
\]
Moreover, when the structure $m$ is compared with the formal
enhancement, the secondary obstruction set reduces to the singleton
\[
\mathcal K_4(m,0)
=
\{[m_4]\}
\subseteq
\operatorname{Harr}^{4,-2}
\bigl(H^*(\widetilde M;\mathbb Q),
H^*(\widetilde M;\mathbb Q)\bigr).
\]
By Theorem~\ref{thm:finite}, vanishing of this class would make the
enhancement isotopic to the formal one. The non-formality of
$\widetilde M$ therefore implies
\[
[m_4]\neq0,
\qquad\text{equivalently}\qquad
0\notin\mathcal K_4(m,0).
\]
Using the Poincar\'e pairing and the cyclic--Harrison comparison, this
is equivalently expressed as
\[
[Tm_4]\neq0
\quad\text{in}\quad
HC_{\mathrm{Harr}}^{5,-2}
\bigl(H^*(\widetilde M;\mathbb Q)\bigr),
\]
with the convention on the second grading adopted in the preceding
subsubsection on cyclic Harrison cohomology.

Thus the Fern\'andez--Mu\~noz manifold lies in the stratum with
vanishing primary obstruction and nonvanishing secondary obstruction.
It is therefore distinguished from the formal enhancement of the
same Poincar\'e cohomology algebra by its class $[m_4]$, rather than
by a nonzero ternary multiplication.

\section{A new proof of Crowley--Nordstr\"om's formality result}\label{sec:cavalcanti}

In this section, using the results established in the previous sections, we give an alternative proof of the following result due to Crowley--Nordstr\"om \cite[Theorem 1.14]{CN} (generalizing a previous result by Cavalcanti \cite[Theorem 4]{Cavalcanti2006}). We assume throughout this section that the ground field is $\F = \R$.\footnote{This section was written in collaboration with Domenico Fiorenza.}

\begin{theorem}\label{thm:CCN}
	Let $\Aa^\ast$ be an $(r-1)$-connected Poincar\'e DGCA of degree
	$n=4r-1$, where $r\geq2$, and suppose that $b^r\leq3$. Set
	\[
	H^*\coloneqq H^*(\Aa).
	\]
	If there exists an element $\varphi\in H^{2r-1}$ such that
	\[
	\varphi\cdot-:
	H^r\longrightarrow H^{3r-1}
	\]
	is an isomorphism, then  $\Aa^*$ is intrinsically
	formal. 
\end{theorem}

\begin{proof}
	 By Example~\ref{ex:hodge}(2), $\Aa^\ast$ is weakly equivalent to  a   Poincar\'e DGCA  admitting a
	Hodge homotopy. Thus, without loss of generality, we  assume  that  $\Aa^\ast$  admits a  Hodge homotopy. Homotopy transfer therefore gives a minimal unital
	cyclic $C_\infty$-algebra enhancement
	\[
	m=m_3+m_4+\cdots
	\]
	of $H^* \coloneqq H^*(\Aa^*)$.	
	We shall prove that its primary obstruction class
	\[
	[m_3]\in
	\operatorname{Harr}^{3,-1}(H^*,H^*)
	\]
	vanishes. Since
	\[
	4r-1\leq5r-3
	\qquad (r\geq2),
	\]
	Theorem~\ref{thm:finite} then implies that this $C_\infty$-structure
	is $C_\infty$-isotopic to the formal one. 
	
	\medskip
	
	We shall repeatedly use the degree restriction
	\begin{equation}\label{eq:m3-critical-degree}
		m_3(x,y,z)=0
		\qquad
		\text{unless }x,y,z\in H^r,
	\end{equation}
	which follows from \cite[Lemma~2.18]{FiorenzaLe2025}.

Moreover, by \eqref{eq:vanishmkr},
\[
m_3(x,x,x)=0
\qquad
\text{for every }x\in H^r.
\]

If $b^r\leq1$, multilinearity therefore implies that $m_3$ vanishes
identically.

\subsection{The algebraic curvature tensor of $\Aa^\ast$}
To prove Theorem~\ref{thm:CCN} for $b^r=2$ and $b^r=3$, we require some preparation. For the remainder of this section, $(\Aa^\ast,d,d^-,\varphi)$ is exactly as in the statement of Theorem \ref{thm:CCN}. In particular, the multiplication $m_3$ on $H^\ast$ enjoys the symmetries described in Section \ref{sec:cyclicity}. We define the tensor $R$ as
\begin{align}
R\colon (H^r)^{\otimes 3}&\to H^r\notag\\
 x\otimes y\otimes z&\mapsto \psi \big(m_3 (x, z, y)\big), \label{F-def}
\end{align}
where $\psi\colon H^{3r-1}\xrightarrow{\sim} H^r$ is the inverse of $\varphi\cdot-$. In other words, $R$ is defined implicitly by the equation

\[
\varphi\cdot R(x,y,z)=m_3(x,z,y).
\]

Notice the swapping of $y$ and $z$ on the right-hand side. This convention is motivated by comparing the symmetries of the Bianchi--Massey tensor introduced by Crowley--Nordstr\"om \cite{CN} with the symmetries of the ``curvature'' tensor $R$ in \eqref{eq:curvature}.
Using the nondegenerate pairing $\la -, -\ra_\varphi$ on $H^r$, we can view $R$ as a multilinear map $(H^r)^{\otimes 4}\to \R$. With this perspective, we write $R_{ijkl}$ for the structure constants of $R$ with respect to a chosen linear basis $(e_i)$ of $H^r$:

\[
R_{ijkl}=\langle R(e_i,e_j,e_k),e_l\rangle_\varphi=\int \varphi\cdot R(e_i,e_j,e_k)\cdot e_l =\int m_3(e_i,e_k,e_j)\cdot e_l,
\]
which is equivalent to:
\begin{equation}R_{ijkl}=\langle m_3(e_i,e_k,e_j),e_l\rangle.\label{eq:curvature}
\end{equation}

\begin{lemma}\label{lemma:cyclicity-R} We have the relation:
\begin{equation}\label{eq:cyclicity-R}
	R_{ijkl}=-(-1)^rR_{klji}.
\end{equation}
\end{lemma}
\begin{proof} By the cyclic identity \eqref{eq:cyclicsign}, we have
\[
R_{ijkl}=\langle m_3(e_i,e_k,e_j), e_l \rangle = - (-1)^r \langle m_3(e_k, e_j,e_l), e_i \rangle=-(-1)^rR_{klji}.\]
\end{proof}

\begin{lemma}[Symmetries of $R$]\label {lem:ric} 
\hfill\begin{enumerate}\item[a)] \textbf{Assume $r$ is even.} The tensor $R$ satisfies:
\begin{enumerate}\item[(1)] $R_{ijkl} = -R_{jikl}$ (Antisymmetry in the first two indices)\item[(2)] $R_{ijkl} = -R_{ijlk}$ (Antisymmetry in the last two indices)\item[(3)] $R_{ijkl} = R_{klij}$ (Pair-swapping symmetry)
\end{enumerate}
Therefore, $R$ defines a linear map $S^2(\wedge^2 H^r) \to \R$. If $\dim H^r\leq 3$, then $R$ also satisfies the Bianchi identity and is therefore a valid algebraic curvature tensor.    
\item[b)] \textbf{Assume $r$ is odd.} The tensor $R$ satisfies:
		\begin{enumerate}
			\item[(1)] $R_{ijkl}= R_{jikl}$ (Symmetry in the first two indices)
			\item[(2)] $R_{ijkl}= R_{ijlk}$ (Symmetry in the last two indices)
			\item[(3)] $R_{ijkl}= R_{klij}$ (Pair-swapping symmetry)
		\end{enumerate}
		Therefore, $R$ defines a linear map $S^2(S^2 H^r) \to \R$.
	\end{enumerate}
\end{lemma}

\begin{proof} a) Assume that $r$ is even. To prove (1), we must show that
	\[
	\int m_3(x,z,x)\cdot w = 0. 
	\]
	This is an immediate consequence of  \eqref{eq:symmetry2e}. To prove (2), we utilize Lemma \ref{lemma:cyclicity-R} along with (1) to obtain
	\[
	R_{ijkl}\stackrel{(1)}{=}-R_{jikl}\stackrel{\eqref{eq:cyclicity-R}}{=}R_{klij}\stackrel{\eqref{eq:cyclicity-R}}{=}-R_{ijlk}.
	\]
	To prove (3), we apply Lemma \ref{lemma:cyclicity-R} followed by (2):
	\[
	R_{ijkl} \stackrel{\eqref{eq:cyclicity-R}}{=} -R_{klji} \stackrel{(2)}{=} R_{klij}.
	\]

	Finally, it is a well-known result (see, e.g., \cite[Chapter 1, \S G,1.107]{Besse1987}) that in dimension 3 or less, any tensor $R$ satisfying symmetries (1)--(3) automatically satisfies the Bianchi identity.

	b) Assume that $r$ is odd.
	
	(1) The symmetry in the first two indices follows from Lemma \ref{lem:oddsymmetry}.
	
	(2) The symmetry in the last two indices follows from \eqref{eq:cyclicity-R}, taking into account the symmetry of the first two indices.
	
	(3) By the same logic utilized in the proof of the even case, the cyclicity combined with the first two symmetries yields the pair-swapping symmetry.

\end{proof}

\underline{Case a: $r$ even.} We will assume $r$ is even and $\dim H^r\leq 3$. Because $r$ is even, the nondegenerate pairing $\langle-,-\rangle_\varphi$ is symmetric. For emphasis, and to establish a more direct correspondence with standard notation in Riemannian geometry, we will write $g=\langle-,-\rangle_\varphi$.
By interpreting $R$ as a linear map

\begin{align*}
R\colon H^r\otimes H^r&\to \mathrm{End}(H^r), &  x\otimes y\mapsto \{z\mapsto R(x,z,y)\}
\end{align*}

we can define the corresponding \emph{Ricci tensor}
\begin{equation}\label{eq:ricci}\mathrm{Ric} \coloneqq \mathrm{tr}(R)\colon H^r\otimes H^r\to \R.
\end{equation}

If $(e_i)$ is a linear basis of $H^r$, then
\[
\mathrm{Ric}_{ij}=g^{kl}R_{ikjl},
\]
where we define $g_{kl}=g(e_k,e_l)=\langle e_k,e_l\rangle_\varphi$ and let $g^{kl}$ be its inverse metric tensor. Since, under our assumptions, $R$ is an algebraic curvature tensor, it satisfies the Bianchi identity. The standard argument from Riemannian geometry dictates that $\mathrm{Ric}$ is symmetric. Indeed, the usual contraction of the first Bianchi identity with
$g^{kl}$, together with the curvature symmetries, gives
\[
\mathrm{Ric}_{ij}=\mathrm{Ric}_{ji}.
\] Moreover, since the Weyl tensor vanishes in dimensions 2 and 3, an algebraic curvature tensor in these dimensions is entirely determined by its Ricci tensor. Consequently, the multiplication $m_3$ is completely governed by $\mathrm{Ric}$.
Notice that $m_3$ is a degree $-1$ Harrison 3-cocycle on $H^\ast$, and that $\mathrm{Ric}\cdot \varphi$ serves as a component of a degree $-1$ Harrison 2-cochain. Remarkably, as we will demonstrate in the following subsections, the exact mechanism by which $\mathrm{Ric}$ determines $m_3$ is by realizing $m_3$ as the Harrison coboundary of a specifically constructed extension of $\mathrm{Ric}\cdot \varphi$. This structural fact guarantees that the Harrison cohomology class of $m_3$ vanishes, thereby establishing that $\Aa^\ast$ is formal.

\underline{Case a.2}.  Now we shall   show that $m_3$  is a   Harrison coboundary  in the case $r$  is even  and $b^r=2$.

Let  $\phi  \in C^{2, -1} (H^*, H^*)$ be  Harrison 2-cochain. Assume $\phi(x,y)=0$ unless $(\deg x,\deg y)=(0,0)\pmod r$.  Then  we have
\[
\delta\phi(x, y, z) = x \phi(y, z) - \phi(xy, z) + \phi(x, yz) - \phi(x, y) z.
\] 
Recalling that $m_3(x,y,z)=0$ unless $\deg (x,y,z)=(r,r,r)$, we see that the equation
\[
\delta\phi=m_3
\]
is equivalent to
\[
\begin{aligned}
&x\phi(y,z)-\phi(xy,z)+\phi(x,yz)-\phi(x,y)z\\
&\qquad=
\begin{cases}
	m_3(x,y,z), &\text{if $\deg(x,y,z)=(r,r,r)$},\\
	0, &\text{otherwise}.
\end{cases}
\end{aligned}
\]
and so, since the pairing $\langle-,-\rangle$ is nondegenerate, to
\begin{equation}\label{eq:coboundary}
\begin{aligned}
&\bigl\langle
x\phi(y,z)-\phi(xy,z)+\phi(x,yz)-\phi(x,y)z,w
\bigr\rangle\\
&\qquad=
\begin{cases}
\langle m_3(x,y,z),w\rangle,
&\text{if $\deg(x,y,z,w)=(r,r,r,r)$},\\
0, &\text{otherwise}.
\end{cases}
\end{aligned}
\end{equation}
Recall that in dimension 2, an algebraic curvature tensor is expressed in terms of the metric and of its scalar curvature through the Kulkarni--Nomizu product as
\[
R = \frac{s}{4} g \mathbin{\bigcirc\mspace{-15mu}\wedge} g,
\]
i.e.,
\begin{equation}\label{eq:KN-dim2}
	R_{ijkl}=\frac{s}{2}(g_{ik}g_{jl}-g_{il}g_{jk}),
\end{equation}
where
\[
s=g^{ij}\mathrm{Ric}_{ij}
\]
is the scalar curvature of $R$.
Analyzing Equation \eqref{eq:coboundary},  taking into account  the Lefschetz  isomorphism:  $\varphi\cdot-:H^r\to H^{3r-1}$,  we use the ansatz  that  $\phi (x, y)  = \lambda (x, y) \varphi$ for  $x, y \in H^r$. Taking  into account  degree restrictions  by Lemma  \ref{lem:degree}, we further use the ansatz  that 
$ \phi(a,b) = 0$ if $\deg(a,b)\not\in \{(r,r),(r,2r),(2r,r)\}$.    Using   \eqref{eq:KN-dim2},  we are led to    the following.
\begin{proposition}\label{prop:coboundart-dim2}
	Let   $\phi \in C^{2, -1} ( H^*, H^*)$
	be the Harrison cochain defined  by the following formulas:
	\begin{align*}
		\phi(x,y)&\coloneqq-\frac{s}{4}g(x,y)\cdot \varphi \qquad \text{if $x,y\in H^r$}\\
		\phi(x,v)&\coloneqq-\frac{s}{4} \langle v, \varphi\rangle x\cdot  \varphi \qquad \text{if $x\in H^r, \, v\in H^{2r}$}\\
		\phi(v,y)&\coloneqq-\frac{s}{4}\langle v, \varphi\rangle  y\cdot \varphi  \qquad \text{if $y\in H^r,\, v\in H^{2r}$}\\
		\phi(a,b)&\coloneqq0 \qquad \text{if $\deg(a,b)\not\in \{(r,r),(r,2r),(2r,r)\}$} \\
	\end{align*}
	Then $\delta \phi =m_3$.
\end{proposition}

\begin{proof}
	We consider the case when $\deg(x,y,z,w)=(r,r,r,r)$ first. By linearity, we only need consider the case $(x,y,z,w)=(e_i,e_k,e_j,e_l)$. In this case we have to prove that
	\begin{equation}\label{eq:coboundary-ijkl}
		\langle e_i \phi(e_k, e_j) - \phi(e_ie_k, e_j) + \phi(e_i, e_ke_j) - \phi(e_i, e_k) e_j,e_l\rangle=R_{ijkl},
	\end{equation}
	but this is precisely equation \eqref{eq:KN-dim2}. To address the case when $\deg(x,y,z,w)\neq (r,r,r,r)$, we notice that by degree reasons, the left hand side of \eqref{eq:coboundary} vanishes unless $\deg x+\deg y+ \deg z+\deg w=4r$. Since $\Aa^\ast$ is $(r-1)$-connected, if none between $x,y,z,w$ is in degree zero then they all have degree at least $r$, and so if $\deg(x,y,z,w)\neq (r,r,r,r)$ we have $\deg x+\deg y+ \deg z+\deg w>4r$ and the left hand side of \eqref{eq:coboundary} vanishes. When one among $x,y,z$ has degree zero (and so by linearity we may assume it is the element $1$), an immediate direct computation shows that the left hand side of \eqref{eq:coboundary} vanishes. We are therefore left with considering the case $\deg w=0$, i.e., by linearity $w=1$, with $x,y,z$ of degree at least $r$ and $\deg x+\deg y+\deg z=4r$. Since $\phi$ vanishes unless its arguments have degrees in $\{(r,r),(r,2r),(2r,r)\}$ this immediately restricts the possibilities for $\deg(x,y,z)$ in order to have a possibly nonzero left hand side to $(2r,r,r)$, $(r,2r,r)$ and $(r,r,2r)$. In the first case our equation reduces to
	\[
	\int\left(-\frac{s}{4}g(y,z)\cdot x\cdot  \varphi +\frac{s}{4}\langle x,\varphi\rangle y\cdot\varphi \cdot z\right)=0,
	\]
	which is trivially verified. The other two cases are analogous.
	
\end{proof}

\underline{Case  a.3}.  Now we shall  show that $m_3$  is a Harrison coboundary  in the case that $r$ is even and $b^r=3$.
 We use the fact that in dimension 3 an algebraic curvature tensor $R$ is expressed in terms of the metric, of its Ricci tensor, and of the scalar curvature, through the   Kulkarni--Nomizu product as
\[
R = \left( \mathrm{Ric} - \frac{s}{4}g \right) \mathbin{\bigcirc\mspace{-15mu}\wedge} g,
\]
i.e., 
\begin{equation}\label{eq:KN-dim3}
	R_{ijkl}=g_{ik}\mathrm{Ric}_{jl}-g_{il}\mathrm{Ric}_{jk}+g_{jl}\mathrm{Ric}_{ik}-g_{jk}\mathrm{Ric}_{il}-\frac{s}{2}\left(g_{ik}g_{jl}-g_{il}g_{jk}\right).
\end{equation}
Let
\[
\eta\colon H^r\to H^r
\]
be the linear endomorphism defined by 
\[
\mathrm{Ric}(x,y)=g(\eta(x),y).
\]
In terms of a linear basis $(e_i)$ of $H^r$, this is
\[
\eta_i^j=\mathrm{Ric}_{ik}g^{kj}.
\]

\begin{proposition}\label{prop:a3}
	Let $\phi \in C^{2, -1} (H^*, H^*)$
	be the Harrison cochain defined as follows:
	\begin{align*}
		\phi(x,y)&=\left(\frac{s}{4}g(x,y)-\mathrm{Ric}(x,y)\right)\cdot \varphi \qquad \text{if $x,y\in H^r$}\\
		\phi(x,v)&= \langle v, \varphi\rangle(\frac{s}{4} x-\eta(x))\cdot  \varphi \qquad \text{if $x\in H^r, \, v\in H^{2r}$}\\
		\phi(v,y)&=\langle v, \varphi\rangle (\frac{s}{4} y-\eta(y))\cdot \varphi  \qquad \text{if $y\in H^r,\, v\in H^{2r}$}\\
		\phi(a,b)&=0 \qquad \text{if $\deg(a,b)\not\in \{(r,r),(r,2r),(2r,r)\}$} \\
	\end{align*}
	Then $\delta \phi = m_3$.
\end{proposition}
\begin{proof}
As in the proof of Proposition \ref{prop:coboundart-dim2}, we first
consider $\deg(x,y,z,w)=(r,r,r,r)$. The desired identity is then
\eqref{eq:coboundary-ijkl}, which is precisely \eqref{eq:KN-dim3}.

If $\deg(x,y,z,w)\neq(r,r,r,r)$, the same degree argument reduces the
verification to $w=1$ and
\[
\deg(x,y,z)\in\{(2r,r,r),(r,2r,r),(r,r,2r)\}.
\]
For $\deg(x,y,z)=(2r,r,r)$, the equation reduces to
	\[
	\langle x \phi(y, z) - \phi(xy, z) + \phi(x, yz) - \phi(x, y) z,w\rangle=0
	\]
	\[
\begin{aligned}
0={}&\int\Bigl(
\bigl(\tfrac{s}{4}g(y,z)-\mathrm{Ric}(y,z)\bigr)x\varphi\\
&\hspace{32mm}
-\langle x,\varphi\rangle
\bigl(\tfrac{s}{4}y-\eta(y)\bigr)\varphi z
\Bigr).
\end{aligned}
\]
	which is trivially verified. The other two cases are analogous.
	
\end{proof}

Propositions~\ref{prop:coboundart-dim2} and~\ref{prop:a3} show that
$[m_3]=0$ in both even-dimensional cases. Hence Theorem~\ref{thm:finite}
implies that $\Aa^*$ is formal when $r$ is even.

\underline{Case b: $r$ odd.}  In this final case, $b^r$ must be even because the bilinear form $\la -, -\ra_\varphi$ induces a non-degenerate symplectic form on $H^{r}$. Under our dimensional bounds, this strictly requires $b^r = 2$. Consequently, $\dim(H^r\odot H^r)=\dim(H^r\wedge H^r)=1$. Since the pairing $\langle-,-\rangle_\varphi$ is nondegenerate, the induced multiplication mapping $H^r\odot H^r\to H^{2r}$ is nonzero. Because $\dim(H^r\odot H^r)=1$, any nonzero linear map out of $H^r\odot H^r$ is automatically injective. Thus, $\ker(\cdot\colon H^r\odot H^r\to H^{2r})=0$.
We now invoke \cite[Theorem 3.10]{FKLS2021}, which establishes that if $\ker(\cdot\colon H^r\odot H^r\to H^{2r})=0$, then the small quotient algebra $\mathcal{Q}_{\mathrm{small}}^*$ of $\Aa^\ast$ satisfies $\mathcal{Q}_{\mathrm{small}}^k = H^k$ in the critical degrees $k = 2r-1$ and $k = n-2r+1$. Hence, the same theorem ensures that
$\mathcal{Q}_{\mathrm{small}}^*\cong H^*$, so the small quotient and,
therefore, the full algebra $\Aa^*$ are formal.

Thus $\Aa^*$ is formal in every case. Since the argument uses only the
Poincar\'e algebra $H^*$ and the existence of the Lefschetz class
$\varphi$, it applies to every Poincar\'e DGCA having this cohomology
algebra. Therefore $H^*$, and equivalently $\Aa^*$ in the terminology of
the theorem, is intrinsically formal.
\end{proof}
	
\section{A borderline extension of the Fiorenza--L\^e vanishing theorem}
\label{sec:cavalcantiex}

Fiorenza--L\^e proved in
\cite[Theorem~2.12]{FiorenzaLe2025} that if an
$(r-1)$-connected Poincar\'e DGCA of degree $n$ admits a Hodge
homotopy and
\[
n\leq 5r-3,
\]
then the operations $m_k$ of the transferred minimal unital
$C_\infty$-algebra vanish for every $k\geq4$. In this section we
extend this conclusion to the borderline dimension $n=5r-2$ under
the additional assumption $b^r\leq2$.

\begin{theorem}\label{thm:m4}
	Let $(\Aa^*,d)$ be an $(r-1)$-connected Poincar\'e DGCA over $\Q$
	of degree
	\[
	n\leq5r-2,
	\]
	where $r\geq2$, and suppose that $\Aa^*$ admits a Hodge homotopy
	$d^-$ and that
	\[
	b^r\leq2.
	\]
	Then the operations of the minimal unital cyclic $C_\infty$-algebra
	transferred to $H^*(\Aa)$ via $d^-$ satisfy
	\[
	m_k=0
	\qquad\text{for every }k\geq4.
	\]
\end{theorem}

\begin{remark}\label{rem:m4}
It is worth  mentioning  a related 
formality   result by Cavalcanti \cite[Theorem 4]{Cavalcanti2006}  who showed  the formality of closed smooth manifolds of dimension $4r$   with $b_r\le 2$     assuming the existence of a   Lefschetz   isomorphism as in Theorem \ref{thm:CCN}.
\end{remark}

\begin{proposition}\label{prop:m4}
	Let $(\Aa^*,d)$ be an $(r-1)$-connected Poincar\'e DGCA of degree
	$n\leq 5r-2$ over $\Q$, where $r\geq2$, and assume that $\Aa^*$
	admits a Hodge homotopy. If
	\[
	b^r=\dim H^r(\Aa)\leq2,
	\]
	then the arity-$4$ operation $m_4$ of the transferred minimal unital
	cyclic $C_\infty$-algebra on $H^*(\Aa)$ vanishes.
\end{proposition}
For the proof of Proposition \ref{prop:m4} we shall  use the following.
	
\begin{lemma}\label{lem:rrr}\cite[Lemma 2.16]{FiorenzaLe2025}
	Let $r\geq2$, and let $\Aa^\ast$ be an $(r-1)$-connected
Poincar\'e DGCA of degree $n$ admitting a Hodge homotopy. Let
$m=(m_3,m_4,\ldots)$ be the transferred minimal unital
$C_\infty$-algebra enhancement of $H^*\coloneqq H^*(\Aa^*)$. If
$n\leq\ell(r-1)+3$, then
$m_{\ell-1}(\alpha_1,\ldots,\alpha_{\ell-1})$ is zero unless
\[
\deg(\alpha_1,\ldots,\alpha_{\ell-1})=(r,\ldots,r).
\]
Moreover, for every $\alpha\in H^r$,
\[
m_{\ell-1}(\alpha,\ldots,\alpha)=0.
\]
\end{lemma}

\begin{proof}[Proof of Proposition \ref{prop:m4}]
	By \cite[Theorem~2.12]{FiorenzaLe2025}, the conclusion already holds
	when
	\[
	n\leq5r-3.
	\]
	It therefore remains only to consider the borderline case
	\[
	n=5r-2.
	\]
	
	Applying Lemma~\ref{lem:rrr} with $\ell=5$, we obtain
	\begin{equation}\label{eq:m4-support}
		m_4(\alpha_1,\alpha_2,\alpha_3,\alpha_4)=0
	\end{equation}
	unless
	\[
	\alpha_1,\alpha_2,\alpha_3,\alpha_4\in H^r.
	\]
	The same lemma also gives
	\begin{equation}\label{eq:m4-diagonal}
		m_4(x,x,x,x)=0
		\qquad
		\text{for every }x\in H^r.
	\end{equation}
	
	We shall also use the following two identities:
	\begin{align}
		m_4(x,x,x,y)
		&=
		-m_4(y,x,x,x),
		\label{eq:sym1_proof}\\
		m_4(x,y,y,z)
		&=
		-m_4(z,y,y,x),
		\label{eq:sym2_proof}
	\end{align}
	for all $x,y,z\in H^r$. If $r$ is even, they follow from
	Lemma~\ref{lem:evensymmetry4}. If $r$ is odd, they follow from
	Lemma~\ref{lem:oddsymmetry} with $k=4$, since
	$(-1)^{k-1}=-1$. The corresponding identities were proved there for
	the recursively defined operations $\widehat m_4$ and therefore also
	hold for
	\[
	m_4=\pi_{\Hh^*}\widehat m_4.
	\]
	Notice that these identities do not require any restriction on
	$b^r$.
	
	For $x_1,\dots,x_5\in H^r$, define
	\begin{equation}\label{eq:Fhat4}
		\widehat F(x_1,x_2,x_3,x_4,x_5)
		\coloneqq
		\left\langle
		m_4(x_1,x_2,x_3,x_4),x_5
		\right\rangle,
	\end{equation}
	where $\langle-,-\rangle$ is the Poincar\'e pairing on $H^*(\Aa)$.
	Since
	\[
	|m_4(x_1,x_2,x_3,x_4)|=4r-2
	\]
	and
	\[
	(4r-2)+r=5r-2=n,
	\]
	the pairing
	\[
	H^{4r-2}\otimes H^r\longrightarrow\Q
	\]
	is nondegenerate.
	
	Cyclicity of $m_4$ gives
	\begin{equation}\label{eq:Fhat-cyclic}
		\widehat F(x_2,x_3,x_4,x_5,x_1)
		=
		\widehat F(x_1,x_2,x_3,x_4,x_5).
	\end{equation}
	Indeed, the sign in \eqref{eq:cyclicsign} is
	\[
	(-1)^4(-1)^{r(4r)}=1.
	\]
	
	If $b^r=0$, then \eqref{eq:m4-support} immediately gives $m_4=0$.
	If $b^r=1$, choose a basis element $a$ of $H^r$. Every value of
	$m_4$ on $(H^r)^{\otimes4}$ is a scalar multiple of
	$m_4(a,a,a,a)$, which vanishes by \eqref{eq:m4-diagonal}.
	Hence $m_4=0$ also in this case.
	
	Assume from now on that
	\[
	b^r=2,
	\]
	and choose a basis $(a,b)$ of $H^r$. By multilinearity and the
	nondegeneracy of the Poincar\'e pairing, it is enough to prove that
	\[
	\widehat F(u_1,u_2,u_3,u_4,u_5)=0
	\qquad
	\text{for all }u_i\in\{a,b\}.
	\]
	This is the precise point at which the hypothesis $b^r\leq2$ enters
	the proof. Up to interchanging $a$ and $b$ and applying a cyclic
	permutation, every word of length five in the alphabet $\{a,b\}$ is
	of one of the following four types:
	\[
	aaaaa,\qquad
	aaaab,\qquad
	aaabb,\qquad
	aabab.
	\]
	We treat these four types separately.
	
	\medskip
	
	\noindent
	\emph{Type \(aaaaa\).}
	By \eqref{eq:m4-diagonal},
	\[
	\widehat F(a,a,a,a,a)
	=
	\langle m_4(a,a,a,a),a\rangle
	=
	0.
	\]
	
	\medskip
	
	\noindent
	\emph{Type \(aaaab\).}
	Apply \eqref{eq:m4-diagonal} to $a+tb$ and take the coefficient of
	$t$. We obtain
	\[
	\begin{aligned}
		0={}&
		m_4(b,a,a,a)
		+m_4(a,b,a,a)\\
		&+
		m_4(a,a,b,a)
		+m_4(a,a,a,b).
	\end{aligned}
	\]
	Pairing this identity with $a$ gives
	\[
	\begin{aligned}
		0={}&
		\widehat F(b,a,a,a,a)
		+\widehat F(a,b,a,a,a)\\
		&+
		\widehat F(a,a,b,a,a)
		+\widehat F(a,a,a,b,a).
	\end{aligned}
	\]
	By cyclicity \eqref{eq:Fhat-cyclic}, all four summands are equal to
	\[
	\widehat F(a,a,a,a,b).
	\]
	Since the ground field has characteristic zero,
	\[
	4\widehat F(a,a,a,a,b)=0
	\]
	implies
	\[
	\widehat F(a,a,a,a,b)=0.
	\]
	Interchanging $a$ and $b$ and using cyclicity gives the vanishing of
	every component of type \(4+1\).
	
	\medskip
	
	\noindent
	\emph{Type \(aaabb\).}
	Using \eqref{eq:sym1_proof}, we obtain
	\[
	\begin{aligned}
		\widehat F(a,a,a,b,b)
		&=
		\langle m_4(a,a,a,b),b\rangle\\
		&=
		-\langle m_4(b,a,a,a),b\rangle\\
		&=
		-\widehat F(b,a,a,a,b).
	\end{aligned}
	\]
	By cyclicity,
	\[
	\widehat F(b,a,a,a,b)
	=
	\widehat F(a,a,a,b,b).
	\]
	Consequently,
	\[
	\widehat F(a,a,a,b,b)
	=
	-\widehat F(a,a,a,b,b),
	\]
	and hence
	\[
	\widehat F(a,a,a,b,b)=0.
	\]
	This proves the vanishing of every component in which the two
	occurrences of one basis element are adjacent in the cyclic word.
	
	\medskip
	
	\noindent
	\emph{Type \(aabab\).}
	Taking $x=z=b$ and $y=a$ in \eqref{eq:sym2_proof}, we obtain
	\[
	m_4(b,a,a,b)
	=
	-m_4(b,a,a,b).
	\]
	Thus
	\[
	m_4(b,a,a,b)=0.
	\]
	Pairing with $a$ gives
	\[
	\widehat F(b,a,a,b,a)=0.
	\]
	By cyclicity,
	\[
	\widehat F(a,a,b,a,b)
	=
	\widehat F(b,a,a,b,a)
	=
	0.
	\]
	This proves the vanishing of the remaining cyclic type, in which the
	two occurrences of one basis element are separated.
	
	We have now shown that every component of $\widehat F$ with respect
	to the basis $(a,b)$ vanishes. Hence
	\[
	\widehat F=0
	\qquad\text{on }(H^r)^{\otimes5}.
	\]
	By the nondegeneracy of the Poincar\'e pairing
	\[
	H^{4r-2}\otimes H^r\longrightarrow\Q,
	\]
	it follows that
	\[
	m_4(x_1,x_2,x_3,x_4)=0
	\qquad
	\text{for all }x_1,x_2,x_3,x_4\in H^r.
	\]
	Finally, \eqref{eq:m4-support} implies that $m_4$ vanishes on all of
	$(H^*)^{\otimes4}$.
\end{proof}

\begin{remark}
	The identities \eqref{eq:sym1_proof} and \eqref{eq:sym2_proof} hold
	without the hypothesis $b^r\leq2$ and may therefore be useful in
	higher-dimensional situations. The dimension restriction is used
	only in the classification of cyclic basis components into the four
	binary types
	\[
	aaaaa,\qquad aaaab,\qquad aaabb,\qquad aabab.
	\]
	When $b^r\geq3$, additional components involving three or more
	distinct basis elements occur, and the preceding argument does not
	force them to vanish.
\end{remark}

\begin{proof}[Proof of Theorem \ref{thm:m4}]  
	  By \cite[Theorem~2.12]{FiorenzaLe2025} with $\ell=6$, we have
	\[
	m_k=0\qquad(k\geq5),
	\]
	because
	\[
	5r-2\leq6(r-1)+2=6r-4.
	\]
Taking into account  Proposition \ref{prop:m4} we conclude   Theorem \ref{thm:m4}.
	\end{proof}

\section{Conclusion and outlook}\label{sec:final}

We conclude with several consequences of our results and directions for
future work.

\begin{enumerate}
	\item
	For a fixed finite-dimensional graded commutative algebra $H^*$, more
	generally under the finite-dimensional truncation hypothesis of
	Theorem~\ref{thm:iso_from_equiv}, our obstruction theory gives a
	recursive classification of minimal $C_\infty$-algebra enhancements of
	$H^*$ up to isotopy. At the $k$-th stage, the obstruction to extending
	an isotopy modulo $(k-1)$ is encoded by the obstruction set
	\[
	\mathcal K_k(m,m'),
	\]
	and these obstruction sets satisfy the generalized additivity property
	of Theorem~\ref{thm:additivek} and
	Proposition~\ref{prop:obstruction-additivity}. Thus the rational
	homotopy types with fixed cohomology algebra $H^*$ can be distinguished
	recursively by a hierarchy of Harrison obstruction classes.
	
	By \cite[Theorem~4.4]{FiorenzaLe2025}, the primary Harrison class
	\[
	[m_3]\in\operatorname{Harr}^{3,-1}(H^*,H^*)
	\]
	and the Bianchi--Massey tensor of Crowley--Nordstr\"om determine one
	another. Nagy--Nordstr\"om introduced the pentagonal Massey tensor
	$\mathcal P$, a quintic rational homotopy invariant related to
	fourfold Massey products \cite{NN2021}. It would be interesting to
	determine how $\mathcal P$, including its dependence on auxiliary
	choices when the Bianchi--Massey tensor does not vanish, is encoded by
	our relative higher obstruction sets, in particular by the arity-$4$
	obstruction theory.
	
	\item
	Our recursive description complements the global deformation-theoretic
	description of Schlessinger--Stasheff \cite{SS2012}. Their approach
	describes rational homotopy types with fixed cohomology as a moduli
	space obtained from an algebraic variety of perturbations, whereas our
	approach resolves the gauge-equivalence problem successively along the
	filtration of the controlling Gerstenhaber DGLA. This produces
	explicit obstruction classes and affine obstruction sets at each
	finite arity.
	
	\item
	There is a potentially fruitful relation with the
	$\mathrm{IBL}_\infty$ models of equivariant string topology developed
	by Cieliebak,  H\'ajek, and Volkov \cite{CHV2022}. Their work compares
	algebraic and analytic $\mathrm{IBL}_\infty$-structures associated
	with cyclic models of closed oriented manifolds. It would be
	interesting to investigate whether the finite-stage information
	provided by isotopy modulo $k$ controls corresponding finite portions
	of the induced $\mathrm{IBL}_\infty$-structure, and how changes of
	cyclic $C_\infty$-models are reflected in the analytic construction.
	
	\item
	The proof of Theorem~\ref{thm:m4} shows that, in the borderline
	dimension $n=5r-2$ and under the assumption $b^r\leq2$, cyclicity and
	the special symmetries of $m_4$ force the vanishing of the entire
	arity-$4$ operation. It is natural to ask for the optimal dimension
	range in which such a conclusion remains valid. An extension toward
	the range $n\leq6r-4$ would have to account for additional degree
	patterns for $m_4$ and for the higher rational homotopy information
	detected by the pentagonal Massey tensor.
	
	\item
	The deformation DGLA controlling minimal $C_\infty$-enhancements is
	closely related to other Lie-theoretic models in rational homotopy
	theory. Berglund \cite{Berglund2011} and
	Buijs--F\'elix--Murillo \cite{BFM2013} constructed
	$L_\infty$-models for components of mapping spaces. A natural
	direction for future work is to compare the Harrison DGLA controlling
	$C_\infty$-enhancements of $H^*(X)$ with $L_\infty$-models for the
	components of $\operatorname{Map}(X,X)$ consisting of self-homotopy
	equivalences, and to determine how our obstruction classes appear in
	the rational homotopy theory of these components.
	
	\item
	Finally, the construction of isotopy modulo $k$, the obstruction sets
	$\mathcal K_k$, and their additivity are not specific to Harrison
	cochains. They apply to Maurer--Cartan elements in any complete
	filtered pronilpotent DGLA whose differential and bracket are
	compatible with the filtration. The passage from compatible
	equivalences at every finite stage to a genuine gauge equivalence
	requires the corresponding inverse-limit, or Mittag--Leffler,
	hypothesis used in Theorem~\ref{thm:iso_from_equiv}.
\end{enumerate}

\end{document}